\documentclass{amsart}

\usepackage{amssymb, graphicx, amsrefs}
\usepackage{color}
\usepackage{graphicx}
\usepackage{float} 
\usepackage{subfigure}
\usepackage{bbm}
\usepackage{hyperref}
\usepackage{fancyhdr} 

\hypersetup{hypertex=true,
	colorlinks=true,
	linkcolor=blue,
	anchorcolor=blue,
	citecolor=blue}

\copyrightinfo{2022}{Jin Sun}

\newtheorem{theorem}{Theorem}[]
\newtheorem{lemma}[theorem]{Lemma}

\newtheorem{corollary}[theorem]{Corollary}

\theoremstyle{definition}
\newtheorem{definition}[theorem]{Definition}

\theoremstyle{remark}
\newtheorem{remark}[theorem]{Remark}

\numberwithin{equation}{section}

\newcommand{\partiald}[2]{  \frac{\partial #1 }{\partial #2}}
\newcommand{\partialdd}[2]{ \partial #1/ \partial #2}
\newcommand{\derivative}[2]{ \frac{\mathrm{d} #1 }{\mathrm{d} #2}}

\newcommand{\av}[1]{\left\vert #1\right\vert}
\newcommand{\aV}[1]{\left\Vert #1\right\Vert}

\newcommand{\R}{\mathbb{R}}
\newcommand{\RRe }{\mathrm{Re}\,}

\newcommand{\dd}{\,\mathrm{d}}

\providecommand{\eat}[1]{}

\newcommand{\Hm}[1]{\leavevmode{\marginpar{\tiny%
			$\hbox to 0mm{\hspace*{-0.5mm}$\leftarrow$\hss}%
			\vcenter{\vrule depth 0.1mm height 0.1mm width \the\marginparwidth}%
			\hbox to 0mm{\hss$\rightarrow$\hspace*{-0.5mm}}$\\\relax\raggedright
			#1}}}

\begin{document}
	
	\title{Curvature Estimates of Nodal Sets of Harmonic Functions in the Plane}
	
	\author{Jin Sun}
	\address{Jin Sun, School of Mathematical Sciences, Fudan University, 200433, Shanghai, China}
	\email{jsun22@m.fudan.edu.cn}
	
	
	
	\begin{abstract}
		In this paper, we study curvature estimates for nodal sets of harmonic functions in the plane. We prove that at any point $p$, the curvature of each nodal curve of any harmonic function $u$ is upper bounded by
		\begin{eqnarray*}
			\av{\kappa(u)(p)}\leq \frac{4(n+1)}{nr}\cos n\alpha_0, 
		\end{eqnarray*}
		where $u$ has only $n$ nodal curves in $B_r(p)$ intersecting at $p$, and $\alpha_0=0$ for odd $n$ or $\alpha_0=\frac{\pi}{2n(n+1)}$ for even $n$. This result is sharp for all $n\geq 1$. In extreme cases, $u$ can be given by the Poisson extension of Dirac measure and its derivatives. Moreover, the curvature of any nodal curve is uniformly upper bounded at every point in the nodal set of $u$ in a small neighborhood $B_{cr}(p)$, where $c<1$ depends only on $n$. Furthermore, with the frequency tool, we prove that the area of the positive part and the negative part of $u$ have a uniform lower bound, which depends only on the number of nodal domains in $B_r(p)$.
	\end{abstract} 
	\maketitle
	
	\section{Introduction}
	
	The study of nodal sets of harmonic functions is an active research area in the geometric theory of PDEs dating back to  origins of potential theory. 
	Analytically, the length of nodal curves of solutions of elliptic equations in the plane is bounded by area of the region and the gradient of solutions \cite{AG}. For hamonic functions on $\mathbb{R}^n$ ($n\geq2$), Hausdorff measure of nodal sets is bounded by the frequency of harmonic functions \cite{HL}.  Geometrically, in $\R^2$, the nodal set of any harmonic function $u$ at some point $p$ can be locally biholomorphic to $\{z\in\mathbb{C}|Re\,z^n=0\}$, where $n$ is the vanishing order of $u$ at $p$, and the nodal set is composed of some maximal analytic curves \cite{WWZ}, which are called nodal curves. In particular, the nodal set of any harmonic function at any non-critical point $p$ is a nodal curve. Moreover,  \"{U}. Kuran first proved in 1969 that at any non-critical point $p$, the curvature of the nodal curve of a harmonic function is upper bounded and the bound depends only on the distance from $p$ to the boundary and to other nodal curves, see  \cite{K}. De Carli and Hudson gave another proof in \cite{CH}. Stefan Steinerberger discovered the equivalent condition when the equality holds in the estimate \cite{SS}. Some subtle examples of nodal sets of harmonic functions were shown in \cites{F,FDH,EAD}. Besides, topological and geometric properties of nodal sets of elliptical PDEs were given in \cites{SP,BM,KW}.
	
	For hamonic functions on $\mathbb{R}^n$ ($n\geq2$), a positive lower bound for the principal curvature of strict convex nodal sets has been proved in \cite{CAMY}. More convexity results related to solutions of linear or nonlinear elliptical PDEs can be found in  \cites{KB,BBJ,CA,GRM}. There are also some remarkable results about solutions of elliptical PDEs on annuli, see \cites{DNJ,JJMO}. However, the results mentioned above are about non-critical points, while few results related to critical points were proved.
	
	This paper aims to show that in $\R^2$, even at any critical point $p$ of a harmonic function, the curvature of any nodal curve at that point can be bounded, which is the main result of the paper, see Theorem \ref{Main_theorem}. Besides, the extreme case when the curvature attains its maximum is unique to some extent. Moreover, the curvature at any point of any nodal curve in a small neighborhood of $p$ can also be bounded. This geometric property tells us that the nodal set cannot ‘bend very much’, corresponding to the analytic property that the Hausdorff measure of nodal sets is bounded. In addition, using the fact that the frequency of $u$ is bounded by the number of nodal domains and the result proven by Logunov and Malinnikova (\cite{logunov2018nodal}), this geometric property means that the area of the positive part and the negative part of $u$ can be bounded below, giving area estimates. 

	We assume that a harmonic function $u$ is defined on a connected region $\Omega\subset\R^2$, and $u$ is not a constant. The vanishing order of $u$ at a point $p\in\Omega$ is denoted by $n(u,p)$, and there will be $n=n(u,p)$ nodal curves intersecting at $p$. The nodal set of $u$ is denoted by $Z(u)$. We identify $\R^2$ with $\mathbb{C}$, so that in any ball $B_r(p)\subset\Omega$, we can find a holomophic function $w$ satisfying
	$$w(z)=u(z)+iv(z),$$
	where $v(z)$ is the conjugate function of $u$ in $B_r(p)$. 
	
	The partial derivatives $\partialdd{u}{x}$ and $\partialdd{u}{y}$ are abbreviated as $u_x$ and $u_y$, respectively. Let $\kappa(u)|_{\gamma_k}(p)$ be the curvature of some nodal curve $\gamma_k\subset Z(u)$ of $u$ at the point $p$. 
	
	Our main result gives a sharp curvature estimate for $p\in Z(u)$ with vanishing order $n(u,p)$. The main ingredient in the proof of the following main theorem is the precise growth estimates of Fourier coefficients of the boundary value $g=u|_{S^1}$. This is a new method for dealing with curvature estimates, and is a different method from the use of mass transport in the proof of Stefan Steinerberger \cite{SS}, which can be useful only for $n=1$.
	
	For convenience, we introduce a definition:
	\begin{definition}\label{def1}
		We say a distribution $T\in (C^\infty(S^1))^\prime$ almost changes its sign for $2n$ times on $S^1$, if there exists some function $g_0\in C^\infty(S^1)$ with $2n$ zeros, counted with multiplicity, such that for every smooth function $f\geq 0$,
		\begin{eqnarray}\label{change_sign}
			(g_0T)(f) = T(g_0f) \geq 0.
		\end{eqnarray} 
	\end{definition}
	In fact, (\ref{change_sign}) is equivalent to say $g_0T\geq 0$ in the sense of $(C^\infty(S^1))^\prime$. In light of the equivalence of positive distributions and positive measures \cite{strichartz2003guide}, this tells us that there is a positive locally finite measure $\mu$ such that 
	\begin{eqnarray}
		(g_0T)(f) = \int_{S^1} f\dd\mu.
	\end{eqnarray} 
	
	And for the case $T=g\in L^1(S^1)$, (\ref{change_sign}) is equivalent to
	\begin{eqnarray}\label{change_signL}
		\frac{1}{2\pi}\int_{0}^{2\pi} f(\theta)g_0(\theta)g(\theta) \dd\theta\geq 0.
	\end{eqnarray}
	
	Then we derive the following main theorem. 
	\begin{theorem}\label{Main_theorem}
		Let $u:B_{r_0}(0)\subseteq \R^2 \rightarrow \R$ be a non-constant harmonic function with boundary value $u|_{S^1_{r_0}}\in L^1(S^1_{r_0})$, such that $u(0)=0$. Assume $u$ has only $n$ nodal curves that intersect at $0$ in $B_{r_0}(0)$.
		Then for every nodal curve $\gamma_k\subset Z(u)$,
		\begin{eqnarray}\label{e1}
			\bigg|\kappa(u)|_{\gamma_k}(0)\bigg|\leq \frac{4(n+1)}{nr_0}.
		\end{eqnarray}
		This equality is sharp when $n$ is odd. Moreover, when $n$ is even, the upper bound becomes 
		\begin{eqnarray}\label{e2}
			\bigg|\kappa(u)|_{\gamma_k}(0)\bigg|\leq \frac{4(n+1)}{nr_0}\cos\frac{\pi}{2(n+1)},
		\end{eqnarray}
		which is also sharp.
	\end{theorem}
	
	Generally, let $\Omega$ be a region, $u:\Omega\subseteq \R^2 \rightarrow \R$ be a non-constant harmonic function satisfying $u(p)=0$, and $r_0$ be the maximal radius satisfying that, for all $r<r_0$, $u$ has only $n$ nodal curves intersecting at $p$ in $B_r(p)\subset\Omega$. Then the curvature of every nodal curve of $Z(u)$ at $p$ will satisfy (\ref{e1}) or (\ref{e2}), according to $n(u,p)$. Thus, the upper bound depends on $d(p,\partial\Omega)$ and the distance between $p$ and other nodal curves. Therefore, this estimate is a local and precise result.
	
	In fact, we can give a characterization for the equality case in Theorem \ref{Main_theorem}.
	\begin{theorem}\label{extremer_case}
		Assume $u$ is the Poisson extension of a distribution $T$, which almost changes its sign for $2n$ times, and $u$ has $n$ nodal curves intersecting at $0$ in $B_1(0)$. After multiplying a constant and rotation, if $u=\RRe w$ is an extremer of inequality (\ref{e1}) or inequality (\ref{e2}), then,
		\begin{eqnarray}\label{extremers}
			{w}(z)=\frac{z^n }{(1-e^{-i\varphi_0}z)^{2n}}+\frac{2z^{n+1}}{(1-e^{-i\varphi_0}z)^{2n+1}}e^{-i(n+1)\varphi_0}\cos n\varphi_0,
		\end{eqnarray}	
		where $\varphi_0$ is given by 
		\begin{eqnarray*}
			\varphi_0=\frac{k\pi}{n}+\frac{\pi}{2n(n+1)}l,
		\end{eqnarray*}
		where $l=0$ if $n$ is odd and $l=1$ if $n$ is even, and $0\leq k\leq n-1$ is an integer.
	\end{theorem}
	
	Furthermore, with growth estimates of every Fourier coefficient of $u$ on the boundary, our paper proves that the curvature of every nodal curve of $u$ in a small neighborhood of $0$ is bounded, which depends only on the vanishing order $n(u,0)$.  
	\begin{theorem}\label{theorem2}
		Let $u:B_{r_0}(0)\subseteq \R^2 \rightarrow \R$ be a non-constant harmonic function with the boundary value $u|_{S^1_{r_0}}\in L^1(S^1_{r_0})$, such that $u(0)=0$. Assume $u$ has only $n$ nodal curves intersecting at $0$ in $B_{r_0}(0)$.
		Then there is a constant $c(n)<1$ and $C(n)$, depending on $n$, such that
		\begin{eqnarray}\label{e3}
			\bigg|\kappa(u)(z)\bigg|\leq \frac{C(n)}{r_0}, \qquad \forall z\in (Z(u)\backslash\{0\})\cap B_{c(n)r_0}(0).
		\end{eqnarray}
	\end{theorem}
	
	With the growth estimates, we can derive that the doubling index $\mathcal{N}(0,{r})$ of $u$ in a ball $B_r(0)({r}<r_0/2)$ is upper bounded, and the bound depends on $n(u,0)$. Inspired by the conclusion of Logunov and Malinnikova (\cite{logunov2018nodal}), we can prove that the ratio of area of the positive part and area of the negative part of $u$ has upper and lower bounds. 
	\begin{theorem}\label{theorem3}
		Let $u:B_1(0)\subseteq \R^2 \rightarrow \R$ be a non-constant harmonic function satisfying $u(0)=0$ with $u|_{S^1}\in L^1(S^1)$, and $u$ has $N$ nodal domains in $B_1(0)$, then 
		\begin{eqnarray}\label{e4}
			c_0\leq\frac{\av{\{u>0\}\cap B_{r_0}(0)}}{\av{\{u<0\}\cap B_{r_0}(0)}}\leq c_0^{-1},
		\end{eqnarray}
		where $c_0$ is a universal constant and $r_0=r_0(N)<1/2$ depends only on $N$.
	\end{theorem}
	
	In the last section, applying a holomorphic bijection $\psi:B_1(0)\rightarrow B_1(0)$, we will give generalized results of above theorems. That is, $u(0)=0$ can be replaced by $u(p)=0$, where $p\in B_1(0)$ is any fixed point, and we can derive similar results.
	
	\section{preliminaries}
	We first note that locally, $u$ is the real part of some analytic function $w(z)$. After setting $w(p)=u(p)$, the vanishing order $n(u,p)$ is just the vanishing order of $w$ at $p$ (Recall $w'(p)=u_x(p)-iu_y(p)$). So locally, there exists a complex constant $c$, so that $\mathrm{Re}\,c(z-p)^{n(u,p)}$ can be an asymptotic formula of $u(z)$ at $p$. Therefore, there are $n(u,p)$ curves intersecting at $p$, whose tangents divide $2\pi$ equally. Moreover, nodal curves are analytic and can be extended to maximal curves, whose ends extend to $\partial\Omega$ if $u\in C^\infty({\overline{\Omega}})$. See \cite{WWZ} for more details.
	
	When $n(u,p)=1$ in $B_r(p)$, we know that there is only one nodal curve across $p=(x_p,y_p)$. After a certain rotation, we can have $u_x(p)\neq0$ and take $x$ locally as a smooth function about $y$. By the implicit function theorem, if $u_x(x(y),y)\neq 0$, then we obtain a small neighborhood of $p$, in which
	$$x'(y)=-\frac{u_y(x(y),y)}{u_x(x(y),y)}.$$
	Differentiate the above formula, we get
	$$x''(y_p)=\frac{-u_y(p)^2u_{xx}(p)+2u_x(p)u_y(p)u_{xy}(p)-u_x(p)^2u_{yy}(p)}{u_x(p)^3}.$$
	It is well known that the curvature of the graph of a function $x=x(y) \in C^2(\R)$ at a point $p=(x(y),y)$ can be expressed as follow:
	$$\kappa(x)(p)=\frac{x''(y)}{\left(1+[x'(y)]^2\right)^{3/2}}.$$
	So, we have
	\begin{eqnarray}\label{E0}
		\quad\kappa(u)(p)=\frac{-u_y(p)^2u_{xx}(p)+2u_x(p)u_y(p)u_{xy}(p)-u_x(p)^2u_{yy}(p)}{\left(u_x(p)^2+u_y(p)^2\right)^{3/2}}.
	\end{eqnarray}
	In fact, the calculation above does not need $u_x(x(y),y)\neq 0$. 
	
	Through simple calculation and setting $u=\RRe w$, we can easily derive the following formula 
	\begin{eqnarray}\label{E3.3}
		\kappa(u)(p)=\av{w'(p)}\RRe\left(\frac{w''(p)}{(w'(p))^2}\right),
	\end{eqnarray}
	see \cite{RPJ} for more details.
	
	When $n(u,p)\geq2$, the formula (\ref{E3.3}) is meaningless as $w'(p)=0$. However, by Taylor expansion along some nodal curve $\gamma$, we have the following essential estimate:
	\begin{lemma}\label{L3.12}
		Let $u$ be a harmonic function defined in $B_r(0)$, with $u(z)=\RRe w(z)$, and satisfy $u(0)=0, \forall k\leq n-1,\,\nabla^k u(0)=0$ (as a tensor of order $k$), while $\nabla^n u(0)\neq 0$ (where $n=n(u,0)\geq1$ is the vanishing order of u at $0$). Let $\kappa$ be the curvature of some curve $\gamma$ of the nodal set at the origin. Then we have
		\begin{eqnarray}\label{E3.00}
			\av{\kappa(u)(0)}\leq
			\frac{2}{n^2+n}\av{\frac{w^{(n+1)}(0)}{w^{(n)}(0)}}.
		\end{eqnarray}	
		Moreover, the equality holds if and only if $$w^{(n+1)}(0)=\pm e^{-i(n+1)\eta_q}\av{w^{(n+1)}(0)},$$
		where
		\begin{eqnarray}\label{qqqqqq}
			\eta_q=\frac{q\pi+\frac{1}{2}\pi-arg(w^{(n)}(0))}{n},\qquad 0\leq q\leq 2n-1.
		\end{eqnarray}
	\end{lemma}
	\begin{proof}
		By $Cauchy-Riemann$ equations, we have: $\partial w/\partial z=\partial w/\partial z+\partial w/\partial \bar{z}=\partial w/\partial x=u_x-iu_y$, therefore
		$$
		w^{(n)}=\frac{\partial^n w}{\partial z^n}=\partiald{^{n}u}{x^{n}}-i\partiald{^{n}u}{x^{n-1}\partial y}.
		$$ 
		Let $\gamma$ be parametrized by arc length, and $\gamma(0)=0$. By expension along $\gamma$ (i.e. $z=\gamma(t)$), we have
		\begin{eqnarray}\label{E3.10}
			0=\RRe w(\gamma(t))=\RRe \left(\frac{z^n}{n!}w^{(n)}(0)+\frac{z^{n+1}}{(n+1)!}w^{(n+1)}(0)+o(z^{n+1})\right),
		\end{eqnarray}
		where $o(z^{n+1})$ satisfies $\lim\limits_{z\rightarrow0}o(z^{n+1})/z^{n+1}=0$. Thus along $\gamma$:
		\begin{eqnarray*}
			\RRe \left(\frac{z^n}{n!}w^{(n)}(0)\right)=-\RRe \left(\frac{z^{n+1}}{(n+1)!}w^{(n+1)}(0)+o(z^{n+1})\right),\quad z=\gamma(t)\rightarrow 0.
		\end{eqnarray*}
		Specially, $\gamma$ approaches the origin as 
		$$
		\RRe \frac{\gamma(t)^nw^{(n)}(0)}{\av{\gamma(t)}^n}\rightarrow0.
		$$
		In fact, since every nodal curve has two directions when approaching the origin, we only need to choose one of them. By setting $\gamma(t)=r(t)e^{i\theta(t)}$, we get an integer $q$ satisfying $0\leq q\leq 2n-1$ and 
		$$
		\theta(t)=\frac{q\pi+\frac{1}{2}\pi-arg(w^{(n)}(0))}{n}+o(1),\qquad 0\leq q\leq 2n-1.
		$$
		Denoting $\eta_q:=\lim\limits_{t\rightarrow0}\theta(t)$, we have $\theta(t)=\eta_q+o(1)$. So $\eta_q$ is the angle of tangent direction of $\gamma$, and $\pi+\eta_q$ is the angle of the other tangent direction. Thus,
		
		\begin{eqnarray}\label{111}
			\RRe \left(e^{in\eta_q}w^{(n)}(0)\right)=0.
		\end{eqnarray}
		Then $w^{(n)}(0)=\pm ie^{-in\eta_q}\av{w^{(n)}(0)}$.
		
		On the other hand, expand $w'(z)$ and $w''(z)$ at the origin:
		\begin{eqnarray*}
			w'(z)&=&\frac{z^{n-1}}{(n-1)!}w^{(n)}(0)+\frac{z^{n}}{n!}w^{(n+1)}(0)+o(z^{n})\\
			w''(z)&=&\frac{z^{n-2}}{(n-2)!}w^{(n)}(0)+\frac{z^{n-1}}{(n-1)!}w^{(n+1)}(0)+o(z^{n-1}).
		\end{eqnarray*}
		Because every nodal curve is analytic \cite{WWZ}, the curvature of each nodal curve at the origin exists and is continuous. Thus, we only need to take the limit of the curvature to calculate $\kappa(u)(0)$. Combined with equation(\ref{E3.3}) we have ($z=\gamma(t)\rightarrow 0$)
		\begin{eqnarray}\label{E3.0}
			\nonumber&&\kappa(u)(0)
			=\lim\limits_{z\rightarrow0}\av{w'(z)}\RRe \left(\frac{w''(z)}{(w'(z))^2}\right)
			\nonumber=\lim\limits_{z\rightarrow0}\frac{\RRe \left(w''(z)(\overline{w'(z)})^2\right)}{\av{w'(z)}^3}
			\\
			\nonumber&&=\lim\limits_{z\rightarrow0}\left(\frac{\RRe \left(\frac{\bar{z}^n|z|^{2(n-2)}\overline{w^{(n)}(0)}\av{w^{(n)}(0)}^2}{(n-2)!(n-1)!(n-1)!}+
				\frac{2\bar{z}^{n+1}|z|^{2(n-2)}\overline{w^{(n+1)}(0)}\av{w^{(n)}(0)}^2}{(n-2)!(n-1)!n!}\right)
			}
			{\av{\frac{z^{n-1}}{(n-1)!}w^{(n)}(0)}^3}\right.
			\\
			&&+ \left.\nonumber\frac{\RRe \left(\frac{\bar{z}^{n-1}|z|^{2(n-1)}(\overline{w^{(n)}(0)})^2w^{(n+1)}(0)}{(n-1)!(n-1)!(n-1)!}\right)
			}{\av{\frac{z^{n-1}}{(n-1)!}w^{(n)}(0)}^3}\right)\\
			\nonumber&\quad&\left(\text{combine equation (\ref{E3.10}), and note} \  \RRe (\overline{z^nw^{(n)}(0)})=\mathrm{Re}\,(z^nw^{(n)}(0))\right)\\
			\nonumber&&=\lim\limits_{z\rightarrow0}\frac{\RRe \left(
				\frac{(n+2)\bar{z}^{n+1}|z|^{2(n-2)}\overline{w^{(n+1)}(0)}\av{w^{(n)}(0)}^2}{(n-2)!(n-1)!(n+1)!}+
				\frac{\bar{z}^{n-1}|z|^{2(n-1)}(\overline{w^{(n)}(0)})^2w^{(n+1)}(0)}{(n-1)!(n-1)!(n-1)!}
				\right)
			}
			{\av{\frac{z^{n-1}}{(n-1)!}w^{(n)}(0)}^3}.
		\end{eqnarray}
		The infinitesimals above are all $o(z^{3(n-1)})$, so we omit them.
		For the first term in the last equation, we get
		
		\begin{eqnarray*}
			I&=&\lim\limits_{z\rightarrow0}\frac{\RRe \left(
				\frac{(n+2)\bar{z}^{n+1}|z|^{2(n-2)}\overline{w^{(n+1)}(0)}\av{w^{(n)}(0)}^2}{(n-2)!(n-1)!(n+1)!}\right)}
			{\av{\frac{z^{n-1}}{(n-1)!}w^{(n)}(0)}^3}\\
			&=&\lim\limits_{z\rightarrow0}\frac{n^2+n-2}{n^2+n}\frac{\RRe\left(z^{n+1}w^{(n+1)}(0)\right)}{\av{z}^{n+1}\av{w^{(n)}(0)}}\quad (z=\gamma(t))\\
			&=&\frac{n^2+n-2}{n^2+n}\frac{\RRe\left(e^{i(n+1)\eta_q}w^{(n+1)}(0)\right)}{\av{w^{(n)}(0)}}.
		\end{eqnarray*}
		For the second term in the last equation, we have
		\begin{eqnarray*}
			II&=&\lim\limits_{z\rightarrow0}\frac{\RRe \left(\frac{\bar{z}^{n-1}|z|^{2(n-1)}(\overline{w^{(n)}(0)})^2w^{(n+1)}(0)}{(n-1)!(n-1)!(n-1)!}\right)}{\av{\frac{z^{n-1}}{(n-1)!}w^{(n)}(0)}^3}\\
			&=&\lim\limits_{z\rightarrow0}\frac{\RRe \left(\bar{z}^{n-1}(\overline{w^{(n)}(0)})^2w^{(n+1)}(0)\right)}{\av{z}^{n-1}\av{w^{(n)}(0)}^3}\\
			&=&\lim\limits_{t\rightarrow0}\frac{\RRe\left(\overline{\gamma(t)}^{n-1}(\overline{w^{(n)}(0)})^2w^{(n+1)}(0)\right)}{\av{\gamma(t)}^{n-1}\av{w^{(n)}(0)}^3}\\
			&=&\frac{\RRe\left(e^{-i(n-1)\eta_q}(\overline{w^{(n)}(0)})^2w^{(n+1)}(0)\right)}{\av{w^{(n)}(0)}^3}.
		\end{eqnarray*}
		From (\ref{111}), we know
		\begin{eqnarray*}
			II&=&\frac{\RRe\left(e^{i(n+1)\eta_q}w^{(n+1)}(0)(\overline{e^{in\eta_q}w^{(n)}(0)})^2\right)}{\av{w^{(n)}(0)}^3}\\
			&=& -\frac{\RRe\left(e^{i(n+1)\eta_q}w^{(n+1)}(0)\right)}{\av{w^{(n)}(0)}}.
		\end{eqnarray*}
		
		From the equations above, we know that $\kappa(u)(0)$ exists and is finite, and we can obtain
		\begin{eqnarray}\label{000}
			\nonumber\kappa(u)\big|_{\gamma}(0)&=& \left(\frac{n^2+n-2}{n^2+n}-1\right)\frac{\RRe\left(e^{i(n+1)\eta_q}w^{(n+1)}(0)\right)}{\av{w^{(n)}(0)}}\\
			&=& -\frac{2}{n^2+n}\frac{\RRe\left(e^{i(n+1)\eta_q}w^{(n+1)}(0)\right)}{\av{w^{(n)}(0)}}.
		\end{eqnarray}
		With the property that $\av{\RRe z}\leq \av{z}$, $\av{\RRe z}=\av{z}$ if and only if $\mathrm{Im}\ z=0$, the conclusion is established.
	\end{proof}
	
	\begin{remark}\label{R7}
		In fact, every nodal curve $\gamma(t)=r(t)e^{i\theta(t)}$ corresponds to a certain argument $\eta_q, 0\leq q\leq n-1$, the amount of which is just $n=n(u,0)$, implying that  the nodal curves are approximating $\{\RRe (z^nw^{(n)}(0))=0\}$ when $z\rightarrow 0$. Moreover, for different directions of tangent of $\gamma_{q}$,  $\kappa(u)\big|_{\gamma_{q+n}}(0)=-\kappa(u)\big|_{\gamma_{q}}(0)$. And the equation (\ref{000}) gives the calculation formula of $\kappa(u)(0)$.
	\end{remark}
	
	\section{Growth estimates of Fourier coefficients of boundary values}
	
	In this section, we will introduce the main ingredient and calculus in the proof of our main theorem. In the following, we may always regard any periodic function $g(\theta)$ with a period of $2\pi$ as a function $g(e^{i\theta})$ on $S^1$. 
	
	Assume that $w:B_1(0)\rightarrow \mathbb{C}$ is a holomorphic function and has expansion
	\begin{eqnarray}\label{1111}
		w(z) = \sum_{k=1}^{+\infty}a_ke^{i\theta_k}z^k,\quad a_k\in\mathbb{R}.
	\end{eqnarray}
	Therefore, the vanishing order of $u=\RRe w$ at $0$ is $n(u,0)=\min\{k\geq 1|a_k\neq 0\}$. Note that 
	\begin{eqnarray*}
		\widehat{g}(k) = \frac{1}{2\pi}\int_{0}^{2\pi} g(\theta)e^{-ik\theta} \dd\theta = \frac{1}{2\cdot k!}w^{(k)}(0)=\frac{a_k}{2}e^{i\theta_k},
	\end{eqnarray*}
	then $n(u,0)=n$ is equivalent to 
	\begin{eqnarray}\label{vanishing_order}
		\widehat{g}(k) = 0, \forall\ 0\leq k\leq n-1;\ \widehat{g}(n) \neq 0.
	\end{eqnarray}
	\begin{lemma}\label{L1}
		Assume that the boundary value of the harmonic function $u=\RRe f:B_1(0)\rightarrow \mathbb{R}$ is formally written as
		\begin{eqnarray}\label{11112}
			g(\theta) = u(e^{i\theta}) =  \sum_{k=1}^{+\infty}a_k\cos(k\theta + \theta_k),
		\end{eqnarray}
		and $g\in L^1(S^1)$ almost changes its sign for $2n$ times on $S^1$, then
		\begin{eqnarray}\label{11113}
			\av{{a_{k}}} \leq 2^{n+2}nA_nk^{2n} , \quad \forall k\geq 0, 
		\end{eqnarray}
		where $A_n:=\sup_{1\leq l\leq n}\av{a_l}$.
	\end{lemma}
	\begin{proof}
		Without loss of generality, we assume $a_n = 1, \theta_n = 0$. If $g$ almost changes sign for $2n$ times, then $g$ satisfies (\ref{change_signL}). There are $2n$ zero points of $g_0$ on $S^1$, denoted by $\{\iota_j\}_{j=1}^{2n}\in S^1$. If $\cos\frac{\iota_{2j-1}-\iota_{2j}}{2}\in [0,1]$, then we assume $\varphi_j = (\iota_{2j-1}+\iota_{2j})/2$ and $c_j = -\cos\frac{\iota_{2j-1}-\iota_{2j}}{2}\in [-1,0]$, otherwise we assume $\varphi_j = (\iota_{2j-1}+\iota_{2j})/2+\pi$ and $c_j = \cos\frac{\iota_{2j-1}-\iota_{2j}}{2}\in [-1,0]$. Then we may choose
		\begin{eqnarray}
			g_0(\theta) := \prod_{j=1}^{n}{(\cos(\theta-\varphi_j)+c_j)}
		\end{eqnarray}
		to shares the same zeros, where $c_j\in [-1,0]$ for all $j$. Moreover, $g_0(\theta)g(\theta)\geq0, a.e. \forall\theta\in[0,2\pi]$, or $g_0(\theta)g(\theta)\leq 0, a.e. \forall\theta\in[0,2\pi]$. Without loss of generality, we assume $g_0(\theta)g(\theta)$ is non-negative almost everywhere (or we choose $\tilde{g} = -g$). Therefore, for all $f\in C^\infty(S^1)$ and $f\geq 0$, we can always have
		\begin{eqnarray}\label{invariant}
			\frac{1}{2\pi}\int_{0}^{2\pi} f(\theta)g_0(\theta)g(\theta) \dd\theta\geq 0.
		\end{eqnarray}
		
		For convenience, we assume
		\begin{eqnarray*}
			g(\theta) = u(e^{i\theta}) =  \sum_{k=-\infty}^{+\infty}\tilde{a}_ke^{i(k\theta + \tilde{\theta}_k)}, 
		\end{eqnarray*}
		where $\tilde{a}_k = a_{\av{k}}/2, \tilde{\theta}_k = sign(k) \theta_k$ for $k\in \mathbb{Z}$ ($\tilde{\theta}_0:=0$).
		
		Let $f(\theta) = 1 - \cos(k\theta - \varphi_0)$, where $\varphi_0$ is a fixed constant. Thus, 
		\begin{eqnarray}\label{integrate}
			\frac{1}{2\pi}\int_{0}^{2\pi} (1 - \cos(k\theta - \varphi_0))g_0(\theta)g(\theta) \dd\theta\geq 0.
		\end{eqnarray}
		
		From the arbitrariness of $\varphi_0$, we have
		\begin{eqnarray}\label{212}
			\frac{1}{2\pi}\int_{0}^{2\pi} g_0(\theta)g(\theta) \dd\theta&\geq &\sup_{\varphi_0\in \R}\RRe \left(\frac{1}{2\pi}\int_{0}^{2\pi} e^{i(\varphi_0-k\theta)}g_0(\theta)g(\theta) \dd\theta\right)
			\nonumber\\
			&=&\av{\frac{1}{2\pi}\int_{0}^{2\pi} e^{-ik\theta}g_0(\theta)g(\theta) \dd\theta}.
		\end{eqnarray}
		The left hand side of (\ref{212}) can be calculated by
		\begin{eqnarray}\label{213}
			&&\frac{1}{2\pi}\int_{0}^{2\pi} \prod_{j=1}^{n}{(\cos(\theta-\varphi_j)+c_j)}g(\theta) \dd\theta \\
			&&= \frac{1}{2\pi}\int_{0}^{2\pi} \prod_{j=1}^{n}{(\frac{e^{i(\theta-\varphi_j)}+e^{-i(\theta-\varphi_j)}}{2}+c_j)}g(\theta) \dd\theta \nonumber\\
			&&= \frac{1}{2^{n+1}\pi}\sum_{I_1\uplus I_2\uplus I_3=\{1,\ldots,n\}}\int_{0}^{2\pi}\prod_{j_1\in I_1}(2c_j)\prod_{j_2\in I_2}e^{i(\theta-\varphi_{j_2})}\prod_{j_3\in I_3}e^{-i(\theta-\varphi_{j_3})}g(\theta) \dd\theta \nonumber\\
			&&= \frac{1}{2^{n}}\sum_{I_1\uplus I_2\uplus I_3=\{1,\ldots,n\}}\prod_{j_1\in I_1}(2c_j)\prod_{j_2\in I_2, j_3\in I_3}e^{i(\varphi_{j_3}-\varphi_{j_2})}\tilde{a}_{\av{I_3}-\av{I_2}}e^{i\tilde{\theta}_{\av{I_3}-\av{I_2}}} \nonumber\\
			&&\leq \frac{1}{2^{n}}\sum_{j=0}^n2^jC_n^j\sum_{\tilde{j}=0}^{n-j}C_{n-j}^{\tilde{j}}\frac{A_n}{2} = 2^{n-1}A_n. \nonumber
		\end{eqnarray}
		The penultimate inequality follows from that $c_j\in [-1,1]$. 
		
		In the following, for $0\leq m \leq n-1$, we denote $g_m:=\prod_{j=m+1}^{n}{(\cos(\theta-\varphi_j)+c_j)}$ and $\widehat{g_mg}(k) = \frac{1}{2\pi}\int_{0}^{2\pi} e^{-ik\theta}g_m(\theta)g(\theta) \dd\theta$; for $m=n$, $g_m:=1$.
		
		By the same calculation with (\ref{213}), we have 
		\begin{eqnarray}\label{2144}
			\av{\widehat{g_mg}(0)}\leq 2^{n-m-1}A_{n-m}, \av{\widehat{g_mg}(1)}\leq 2^{n-m-1}A_{n-m+1}.
		\end{eqnarray}
		Denoting $A_0:=0$, we know that the above inequality is right for all $0\leq m \leq n$.
		
		Observing that $g_0(\theta) = (\frac{e^{i(\theta-\varphi_1)}+e^{-i(\theta-\varphi_1)}}{2}+c_1)g_1(\theta)$, we can rewrite (\ref{212}) as
		\begin{eqnarray}\label{214}
			\av{e^{i\varphi_1}\widehat{g_1g}(k+1)+2c_1\widehat{g_1g}(k)+e^{-i\varphi_1}\widehat{g_1g}(k-1)}\leq 2^nA_n.
		\end{eqnarray}
		Let $w_m:=(-c_m+\sqrt{1-c_m^2}i)e^{-i\varphi_m}, v_m:=(-c_m-\sqrt{1-c_m^2}i)e^{-i\varphi_m}$ be the roots of
		\begin{eqnarray*}
			{e^{i\varphi_m}z^2+2c_mz+e^{-i\varphi_m}}=0.
		\end{eqnarray*}
		Then (\ref{214}) gives
		\begin{eqnarray*}
			\av{(\widehat{g_1g}(k+1)-w_1\widehat{g_1g}(k))-v_1(\widehat{g_1g}(k)-w_1\widehat{g_1g}(k-1))}\leq 2^nA_n.
		\end{eqnarray*}
		Summing it from $0$ to $k$, and noting $\av{w_1}=\av{v_1}=1$, we derive
		\begin{eqnarray*}
			\av{(\widehat{g_1g}(k+1)-w_1\widehat{g_1g}(k))-v_1^k(\widehat{g_1g}(1)-w_1\widehat{g_1g}(0))}\leq 2^nA_nk.
		\end{eqnarray*} 
		Therefore, with (\ref{2144}),
		\begin{eqnarray}\label{216}
			\av{\widehat{g_1g}(k+1)-w_1\widehat{g_1g}(k)}
			&\leq& 2^nA_nk+\av{\widehat{g_1g}(1)}+\av{\widehat{g_1g}(0)}\\
			&\leq& 2^nA_n(k+1).\nonumber
		\end{eqnarray} 
		Summing it from $0$ to $k-1$, we have
		\begin{eqnarray*}
			\av{\widehat{g_1g}(k)-w_1^k\widehat{g_1g}(0)}
			\leq 2^nA_n\frac{k(k+1)}{2},
		\end{eqnarray*} 
		thus,
		\begin{eqnarray*}
			\av{\widehat{g_1g}(k)}
			\leq 2^nA_n\frac{k(k+1)}{2}+\av{\widehat{g_1g}(0)}\leq 2^{n+1}k^2A_n.
		\end{eqnarray*} 
		By induction, we assume the following inequality holds for some $1\leq m\leq n-1$:
		\begin{eqnarray}\label{217}
			\av{\widehat{g_mg}(k)}
			\leq 2^{n+1}mk^{2m}A_n.
		\end{eqnarray} 
		Observe that $g_m(\theta) = (\frac{e^{i(\theta-\varphi_{m+1})}+e^{-i(\theta-\varphi_{m+1})}}{2}+c_{m+1})g_{m+1}(\theta)$, we can rewrite (\ref{217}) as
		\begin{eqnarray*}
			\av{e^{i\varphi_{m+1}}\widehat{g_{m+1}g}(k+1)+2c_{m+1}\widehat{g_{m+1}g}(k)+e^{-i\varphi_{m+1}}\widehat{g_{m+1}g}(k-1)}\leq 2^{n+1}mk^{2m}A_n.
		\end{eqnarray*}
		Just repeating the above steps, we obtain
		\begin{eqnarray*}
			\av{\widehat{g_{m+1}g}(k+1)-w_{m+1}\widehat{g_{m+1}g}(k)}
			&\leq& \sum_{l=1}^k2^{n+1}ml^{2m}A_n+\av{\widehat{g_{m+1}g}(1)}+\av{\widehat{g_{m+1}g}(0)}\\
			&\leq& \sum_{l=1}^k2^{n+1}ml^{2m}A_n+2^{n-m}A_{n}.
		\end{eqnarray*}
		Summing it from $0$ to $k-1$, we have
		\begin{eqnarray*}
			\av{\widehat{g_{m+1}g}(k)-w_{m+1}^k\widehat{g_{m+1}g}(0)}
			&\leq& \sum_{l=0}^{k-1}(\sum_{\tilde{l}=1}^l2^{n+1}m\tilde{l}^{2m}A_n+2^{n-m}A_{n})\\
			&=& \sum_{l=1}^{k-1}2^{n+1}ml^{2m}(k-l)A_n+2^{n-m}kA_{n} \\
			&\leq& 2^{n+1}(m+1)k^{2(m+1)}A_{n} - 2^{n-m}A_{n},
		\end{eqnarray*} 
		thus, 
		\begin{eqnarray*}
			\av{\widehat{g_{m+1}g}(k)}\leq 2^{n+1}(m+1)k^{2(m+1)}A_{n}.
		\end{eqnarray*} 
		In conclusion, (\ref{217}) holds for all $1\leq m\leq n$. Specially, $g_{n}g=g$ and
		\begin{eqnarray*}
			\av{{a_{k}}} = \av{\frac{1}{k!}w^{(k)}(0)}= 2\av{\widehat{g_{n}g}(k)} \leq 2^{n+2}nA_nk^{2n} , \quad \forall k\geq 1, 
		\end{eqnarray*}
		And $a_0=0$ comes from the $0$th Fourier coefficient of $g$.
	\end{proof}

	\begin{remark}
		In fact, Lemma \ref{L1} tells us that Fourier coefficients of boundary values have polynomial growth, and the growth order is at most $2n$. Moreover, the equality case of inequality (\ref{e1}) or (\ref{e2}) holds when $T=\sum_{j=1}^{2n}b_jD^j\delta_{\varphi_0}$ in the proof of Theorem \ref{extremer_case}, and $\av{a_k}=\av{\frac{1}{\pi}T\left(\zeta^{-k}\right)} = \av{\sum_{j=1}^{2n}b_j(ik)^je^{-ik\varphi_0}}$, meaning that the growth order of Fourier coefficients of $T$ is $2n$. Therefore, the growth order in Lemma \ref{L1} is optimal to some extend.
	\end{remark}
	When the boundary value $g(\theta)=u|_{\partial\Omega}$ of $u$ is continuous, we can observe that $g(\theta)$ will change signs for $2n$ times if $u$ has $n$ nodal curves in $B_1(0)$. Then $g$ will satisfy (\ref{change_signL}). 
	For $g(\theta)\in L^1(S^1)$, according to the result in \cite{ramm2022dirichlet}, we know that $u(r,\theta)=P(r,\cdot)\ast g(\cdot) (\theta)\in C^\infty(B_1(0))$ 
	is the only harmonic function with boundary value $g$, where the Poisson kernel $P(r,\theta) = \frac{1}{2\pi}\frac{1-r^2}{1-2r\cos(\theta)+r^2}$ is an approximate identity on $S^1$ \cite{grafakos2008classical}. Thus, $u(r,\theta)$ converges to $g(\theta)$ in the sense of $L^1$, that is, $\aV{u(r,\theta)-g(\theta)}_{L^1(S^1)}\rightarrow 0$ as $r\rightarrow 1-$. But the condition that $g$ changes sign is no longer meaningful. So we need the definition of almost changing sign, meaning that we can find intervals where $g\geq 0$ (or $g\leq 0$) in the sense of $L^1$. 
	
	In fact, for $g\in L^1(S^1)$, we can also prove that $g$ satisfies (\ref{change_signL}) if $u$ has $n$ nodal curves in $B_1(0)$. The first statement of the following lemma is the extension of results in \cite{WWZ}.
	\begin{lemma}\label{l9}
		Assume $u$ is harmonic in $B_1(0)$ and $g(\theta) = u|_{S^1} \in L^1(S^1)$, and $u$ has only $n$ nodal curves $\{\gamma_k\}_{k=1}^n$ in $B_1(0)$ across the origin, i.e. $\gamma_k(0) = 0$ for all $k$. Then
		
		(i) Every maximal curve $\{\gamma_k(t): -\infty<s_k<t<l_k<+\infty\}\subset B_1(0)$ has a unique extension to the boundary, i.e., $\gamma_k(s_k):=\gamma_k(s_k+)$ and $\gamma_k(l_k):=\gamma_k(l_k-)$ exist and belong to $\partial B_1(0)$;
		
		(ii) Let the set of end points of all maximal curves be $\{\eta_k\}_{k=1}^{2n}$, where $0\leq \eta_1\leq \cdots\leq \eta_{2n}\leq \eta_1+2\pi:=\eta_{2n+1}$. Denote $I_k = [\eta_k, \eta_{k+1}]$, then by multiplying a constant, for all $1\leq k\leq 2n$,
		$$
		(-1)^k g(\theta)\geq 0, \quad a.e.\ \theta\in I_k.
		$$
	\end{lemma}
	\begin{proof}
		(i) By the results of \cite{WWZ}, every nodal curve is a maximal curve in $B_1(0)$, and both ends will extend to the boundary. That is, for any compact set $K\subset B_1(0)$, there exists $\tilde{s}_{k,K}, \tilde{l}_{k,K}$, satisfying that $\gamma_k(s_k,\tilde{s}_{k,K})\cup\gamma_k(\tilde{l}_{k,K},l_k)\subset B_1(0)\backslash K$. Now we need to prove $\gamma_k(s_k+)$ and $\gamma_k(l_k-)$ exist. 
		
		Assume $\gamma_k(s_k+)$ does not exist for some $k$. Then we have two distinct limit points $e^{i\theta_1}$ and $e^{i\theta_2}$ of $\gamma_k(t)$ when $t\rightarrow s_k+$, where $0\leq\theta_1<\theta_2<2\pi$. Using the fact that $\gamma_k$ is an analytic curve and $\gamma_k(t)$ tends to the boundary, we derive that $\gamma_k(t)$ will oscillate infinitely many times on $(\theta_1, \theta_2)$ (or $(\theta_2, \theta_1+2\pi)$) as $t\rightarrow s_k+$. The interval is denoted by $I_0$. Therefore, for all $\theta\in I_0$, $e^{i\theta}$ is a limit point of $\gamma_k$. 
		With the $L^1$ convergence, it follows that when $r\rightarrow 1-$, $u(r,\theta)\rightarrow g(\theta), a.e.\theta\in I_0$. As a consequence of $u|_{\gamma_k}=0$, 
		$$g(\theta)=0, a.e.\theta\in I_0.$$
		We may assume $g(\theta)=0, \forall\theta\in I_0$, which gives no difference for $u(r,\theta)$ from the uniqueness of $u$ \cite{ramm2022dirichlet}. 
		Recall that
		\begin{eqnarray*}
			u(r,\theta) = \int^{2\pi}_0 \frac{1}{2\pi}\frac{1-r^2}{1-2r\cos(\theta-\alpha)+r^2}g(\alpha) \dd\alpha.
		\end{eqnarray*}
		Let $\frac{1}{2}I_0$ denote the concentric interval of $I_0$. Without loss of generality, we assume $I_0 = (\theta_1, \theta_2)$. Then $\forall \theta\in \frac{1}{2}\bar{I}_0, \forall \alpha\in [0,2\pi]\backslash 
		\bar{I}_0$, we have $\av{\theta-\alpha}\geq (\theta_2-\theta_1)/2$, which results in
		\begin{eqnarray*}
			\av{u(r,\theta)} &\leq& \int_{[0,2\pi]\backslash \bar{I}_0} \frac{1}{2\pi}\frac{1-r^2}{1-2r\cos(\theta-\alpha)+r^2}\av{g(\alpha)} \dd\alpha \\
			&\leq& \frac{1}{2\pi}\frac{1-r^2}{1-2r\cos(\frac{\theta_2-\theta_1}{2})+r^2}\int_{[0,2\pi]\backslash \bar{I}_0} \av{g(\alpha)} \dd\alpha.
		\end{eqnarray*}
		So $u(r,\theta)$ tends to $0$ as $r\rightarrow 1-$, implying that $u\in C(B_1(0)\cup S_0)$, where $S_0 = \{(\cos\theta, \sin\theta)|\theta\in \frac{1}{2}I_0\}$. By symmetric principle, we extend $u$ to a harmonic function $\tilde{u}: D\rightarrow\mathbb{R}$, where $\exists \varepsilon<1$,
		\begin{eqnarray*}
			B_1(0)\cup\{(r\cos\theta, r\sin\theta)|\theta\in \frac{1}{2}I_0, r<1+\varepsilon\}\subset D.
		\end{eqnarray*}
		Therefore, for all $\theta\in \frac{1}{2}I_0$, $\tilde{u}$ is analytic along half-line $L_\theta:=\{(r\cos\theta, r\sin\theta)|0<r<1+\varepsilon\}$. But zero points of $\tilde{u}$ on $L_\theta$ have a accumulation point $(\cos\theta, \sin\theta)$, which means $\tilde{u}=0$ on $L_\theta$, for all $\theta\in \frac{1}{2}I_0$. Then $\tilde{u}=0$ in an open set, indicating $u\equiv 0$ in $B_1(0)$, leading to a contradiction. 
		
		Similarly, $\gamma_k(l_k-)$ also exists.
		
		(ii) When $Int(I_k) = \emptyset$, the conclusion is right. Assume $Int(I_k) \neq \emptyset$. As there are only $n$ nodal curves in $B_1(0)$ across the origin, we know that there are $2n$ nodal domains $\{\Omega_k\}^{2n}_{k=1}$ in $B_1(0)$. Without loss of generality, we may assume $\Omega_k\subset \{u>0\}$ and $\overline{\Omega}_k\cap S^1 = \{(\cos\theta, \sin\theta)|\theta\in I_k\}$. Assume $\partial\Omega_k\cap B_1(0) = \gamma_{k_1}(0,l_{k_1})\cup\gamma_{k_2}(0,l_{k_2})\cup\{0\}$ for some $k_1, k_2$. Therefore, by (i), for fixed $\theta\in Int(I_k)$, there exists some $\widetilde{l}_{k_1}>0, \widetilde{l}_{k_2}>0$, such that  $\gamma_{k_1}(\widetilde{l}_{k_1},{l}_{k_1})\subset B_1(0)\cap B_d(\gamma_{k_1}({l}_{k_1}))$ and $ \gamma_{k_2}(\widetilde{l}_{k_2},{l}_{k_2})\subset B_1(0)\cap B_d(\gamma_{k_2}({l}_{k_2}))$, where $d:=\max\{\av{\gamma_{k_1}({l}_{k_1})-e^{i\theta}}, \av{\gamma_{k_2}({l}_{k_2})-e^{i\theta}}\}/2$ . By the compactness of $\gamma_{k_1}[0,\widetilde{l}_{k_1}]$ and $\gamma_{k_2}[0,\widetilde{l}_{k_2}]$, there exists some $r_j:=\sup_{t\in[0,\widetilde{l}_{k_j}]}\av{\gamma_{k_j}(t)}<1, j=1,2$. Letting $r_0:= max\{r_1,r_2\}$, we have $\forall r_0<r<1$, $u(r,\theta)>0$. Then, by the almost everywhere convergence property and arbitrariness of $\theta\in Int(I_k)$, we get that $g(\theta)\geq 0, a.e. \theta\in I_k$.
	\end{proof}
	
	\section{curvature estimates at $0$ with $n(u,0)\geq1$}
	Without loss of generality, we take $r_0=1$ in Theorem \ref{Main_theorem} (otherwise, replace $u(z)$ by $v(z)=u(r_0z)$), where $u$ has only $n\geq 1$ nodal curves intersecting at the origin in $B_1(0)$.~\\
	
	\emph{Proof of Theorem \ref{Main_theorem}}. 
	First, set $u=\RRe w$, and select a conjugate function of $u$ such that $w(0)=0$. Choose any curve $\gamma$ in the nodal set, we just need to prove that the curvature of $\gamma$ at origin is bounded. 
	
	In light of Remark \ref{R7}, since $u$ has $n$ nodal curves across the origin, we know that $n(u,0)=n$ . After a rotation, $w$ has the following expansion:
	$$
	w(z)=a_nz^n+a_{n+1}e^{i\theta_{n+1}}z^{n+1}+\cdots,
	$$
	where $a_n>0$, $w^{(n)}(0)=n!a_n$ and $a_m\in \mathbb{R}$ for all $m\geq n$.
	Thus, formally, 
	$$
	g(\theta) = u(e^{i\theta}) =  \sum_{k=n}^{+\infty}a_k\cos(k\theta + \theta_k),
	$$
	where $\theta_n=0$. Moreover, $\forall k\geq n, \widehat{g}(k) = a_ke^{i\theta_k}/2 = w^{(k)}(0)/(2\cdot k!)$.
	
	By Lemma \ref{L3.12}, we only need to estimate the upper bound of $a_{n+1}/a_n$, that is,
	\begin{eqnarray}
		k(u)\big|_{\gamma_q}(0)&=&-\frac{2}{n^2+n}\frac{\RRe\left(e^{i(n+1)\eta_q}w^{(n+1)}(0)\right)}{\av{w^{(n)}(0)}} \nonumber\\
		&=& -\frac{2}{n}\frac{\RRe\left(e^{i(n+1)\eta_q}\widehat{g}(n+1)\right)}{\widehat{g}(n)}\label{31000} \\
		&=& -\frac{2}{n}\frac{a_{n+1}}{a_n}\cos(\theta_{n+1}+(n+1)\eta_q) \label{31111}.
	\end{eqnarray}
	
	By (ii) of Lemma \ref{l9}, for all $g\in L^1(S^1)$, we know that $g$ almost changes its sign for $2n$ times on $S^1$. Therefore, by Lemma \ref{L1}, we know that $a_{n+1}/a_n$ has an upper bound. Now we aim to figure out a sharp upper bound.
	
	Just as the proof of Lemma \ref{L1}, we have inequality (\ref{invariant}). Take $f(\theta) = 1 - cos(\theta-\varphi_0)$, and note that $\widehat{g}(k)=0, \forall 0\leq k\leq n-1$. Thus, we have the inequality (\ref{212}) for $k=1$. From the same calculation with (\ref{213}), we obtain
	\begin{eqnarray*}
		\frac{1}{2\pi}\int_{0}^{2\pi} g_0(\theta)g(\theta) \dd\theta
		&=&\frac{1}{2\pi}\int_{0}^{2\pi} \prod_{j=1}^{n}{(\cos(\theta-\varphi_j)+c_j)}g(\theta) \dd\theta \\
		&=& \frac{1}{2\pi}\int_{0}^{2\pi} \prod_{j=1}^{n}{(\frac{e^{i(\theta-\varphi_j)}+e^{-i(\theta-\varphi_j)}}{2}+c_j)}g(\theta) \dd\theta \nonumber\\
		&=& \frac{1}{2^{n}}(\widehat{g}(n)e^{i\sum_{j=1}^n\varphi_j}+\widehat{g}(-n)e^{-i\sum_{j=1}^n\varphi_j}) \\
		&=& \frac{a_n\cos(\bar{\varphi})}{2^{n}},
	\end{eqnarray*}
	where $\bar{\varphi}:=\sum_{j=1}^n\varphi_j$.
	
	On the other hand, for some $\varphi_0$, we have 
	\begin{eqnarray}
		\frac{1}{2\pi}\int_{0}^{2\pi} \cos(\theta-\varphi_0)g_0(\theta)g(\theta) \dd\theta
		&=&\RRe \left(\frac{1}{2\pi}\int_{0}^{2\pi} e^{i\varphi_0-i\theta}g_0(\theta)g(\theta) \dd\theta\right)\nonumber\\
		&=&\av{\frac{1}{2\pi}\int_{0}^{2\pi} e^{-i\theta}g_0(\theta)g(\theta) \dd\theta}\nonumber\\
		&=& \av{\frac{1}{2\pi}\int_{0}^{2\pi} \prod_{j=1}^{n}{(\frac{e^{i(\theta-\varphi_j)}+e^{-i(\theta-\varphi_j)}}{2}+c_j)}g(\theta)e^{-i\theta} \dd\theta}\nonumber \\
		&=& \av{\frac{1}{2^{n}}\left(\widehat{g}(n+1)e^{i\bar{\varphi}}
			+2\sum_{j=1}^nc_j\widehat{g}(n)e^{i(\bar{\varphi}-\varphi_j)}\right)}\nonumber \\
		&=& \av{\frac{a_{n+1}e^{i(\theta_{n+1}+\bar{\varphi})}+2\sum_{j=1}^nc_ja_ne^{i(\bar{\varphi}-\varphi_j)}}{2^{n+1}}}\nonumber.
	\end{eqnarray}
	Thus, it follows from (\ref{212}) that
	\begin{eqnarray}
		\frac{1}{2\pi}\int_{0}^{2\pi} e^{i\varphi_0-i\theta}g_0(\theta)g(\theta) \dd\theta
		&=&
		e^{i\varphi_0}\left(\frac{a_{n+1}e^{i(\theta_{n+1}+\bar{\varphi})}+2\sum_{j=1}^nc_ja_ne^{i(\bar{\varphi}-\varphi_j)}}{2^{n+1}}\right)\nonumber\\
		&\leq& \frac{a_n\cos(\bar{\varphi})}{2^{n}}\label{30},
	\end{eqnarray}
	which gives
	\begin{eqnarray}\label{31}
		&&\av{a_{n+1}\cos(\theta_{n+1}+(n+1)\eta_q)}\leq
		\av{a_{n+1}} \\
		&&\leq 2\left(a_n\cos(\bar{\varphi})+\av{\sum_{j=1}^nc_ja_ne^{i(\bar{\varphi}-\varphi_j)}}\right)\leq 2(n+1)\av{a_n}, \nonumber
	\end{eqnarray}
	All equalities hold only if $c_j=-1$ and 
	\begin{eqnarray}
		&&\theta_{n+1}+\bar{\varphi} \equiv \bar{\varphi}-\varphi_j \ (mod\ 2\pi),\ \forall 1\leq j \leq n, \label{32}\\
		&&\theta_{n+1}+(n+1)\eta_q\equiv 0\ (mod\ \pi),\quad \bar{\varphi}
		\equiv 0\ (mod\ 2\pi)\label{33},
	\end{eqnarray}
	where $\bar{\varphi}
	\equiv 0\ (mod\ 2\pi)$ means that there is some $k\in \mathbb{Z}$ such that $\bar{\varphi}=2k\pi$. 
	In light of equation (\ref{30}) and (\ref{32}), we obtain that 
	\begin{eqnarray}\label{onlyif1}
		\varphi_0+\theta_{n+1}+\bar{\varphi} \equiv\varphi_0+ \bar{\varphi}-\varphi_j\ \equiv 0\ (mod\ 2\pi),\ \forall 1\leq j \leq n,
	\end{eqnarray}
	
	For odd $n$ and $q=(n-1)/2$, there exists $\varphi_j$ such that the equalities hold, which will be specified in Lemma \ref{sharpness}. Combining (\ref{31111}) and (\ref{31}), we derive that (\ref{e1}) holds for odd n. 
	
	For even $n$, (\ref{32}) and (\ref{33}) can not hold simultaneously. Otherwise, we may assume $\theta_{n+1}=-\varphi_j$ for all $1\leq j \leq n$ in (\ref{32}). Recall $\overline{\varphi}=\sum_{j=1}^{n}\varphi_j$ and $\eta_q=\frac{q\pi+\frac{1}{2}\pi}{n}$. Thus, by (\ref{33}), we have 
	\begin{eqnarray*}
		-\overline{\varphi}+n(n+1)\eta_q=n\theta_{n+1}+n(n+1)\eta_q\equiv 0\ (mod\ \pi).
	\end{eqnarray*}
	However, we know $n(n+1)\eta_q=(n+1)(q+\frac{1}{2})\pi\equiv \pi/2\ (mod\ \pi)$ and $\overline{\varphi}\equiv 0\ (mod\ 2\pi)$, which leads to a contradiction.
	
	This time, it follows from (\ref{30}) that
	\begin{eqnarray}
		&&\av{{a_{n+1}\cos(\theta_{n+1}+(n+1)\eta_q)+2\sum_{j=1}^nc_ja_n\cos(-\varphi_{j}+(n+1)\eta_q)}}\nonumber\\
		&&\leq
		\av{{a_{n+1}e^{i(\theta_{n+1}+(n+1)\eta_q)}+2\sum_{j=1}^nc_ja_ne^{i(-\varphi_j+(n+1)\eta_q)}}}
		\leq 2{a_n\cos(\bar{\varphi})}.\label{onlyifeven}
	\end{eqnarray}
	
	Let $\iota_j\in \mathbb{Z}, j=0,\ldots, n$, such that $-\varphi_{j}+(n+1)\eta_q+\iota_j\pi\in[-\pi/2,\pi/2]$, where $j=1,\ldots,n$, and $\bar{\varphi}+\iota_0\pi\in[-\pi/2,\pi/2]$. Thus, according to the concavity of $\cos\theta$ on $[-\pi/2,\pi/2]$, we have
	\begin{eqnarray}
		&&\av{{a_{n+1}\cos(\theta_{n+1}+(n+1)\eta_q)}}\label{341}\\
		&&\leq
		\av{{a_{n+1}\cos(\theta_{n+1}+(n+1)\eta_q)+2\sum_{j=1}^nc_ja_n\cos(-\varphi_{j}+(n+1)\eta_q)}} \nonumber\\
		&&\qquad+ \av{2\sum_{j=1}^nc_ja_n\cos(-\varphi_{j}+(n+1)\eta_q)}\nonumber\\
		&&\leq 2\av{a_n}\left(\cos(\bar{\varphi}+\iota_0\pi)+\sum_{j=1}^n\cos(-\varphi_{j}+(n+1)\eta_q+\iota_j\pi)\right)\nonumber\\
		&&\leq 2(n+1)\av{a_n}\cos\left(\frac{\bar{\varphi}+\iota_0\pi+\sum_{j=1}^n(-\varphi_{j}+(n+1)\eta_q+\iota_j\pi)}{n+1}\right)\nonumber\\
		&& = 2(n+1)\av{a_n}\cos\left(\frac{\sum_{j=0}^n\iota_j}{n+1}\pi+n\eta_q\right)\nonumber\\
		&& = 2(n+1)\av{a_n}\cos\left(\frac{\sum_{j=0}^n\iota_j}{n+1}\pi+q\pi+\frac{1}{2}\pi\right)\nonumber\\
		&& \leq 2(n+1)\av{a_n}\cos\left(\frac{\pi}{2(n+1)}\right)\nonumber.
	\end{eqnarray}
	The last inequality follows from the observation that, for all $k\in \mathbb{Z}$,
	\begin{eqnarray*}
		\av{\frac{k}{n+1}+\frac{1}{2}}\geq \frac{1}{2(n+1)}.
	\end{eqnarray*}
	Thus, for even $n$, (\ref{e2}) holds. 
	
	The proof of sharpness follows from lemma \ref{sharpness}.
	\hfill$\square$\par

	\begin{lemma}\label{sharpness}
		For odd $n$, (\ref{e1}) is sharp. For even $n$, (\ref{e2}) is sharp.
	\end{lemma}
	\begin{proof}
		It suffices to find a sequence of $u_k$, such that the curvature of some nodal curve of $u_k$ tends to the upper bound as $k\rightarrow\infty$. Let $\alpha_0:=0$ for odd $n$ and $\alpha_0:=\frac{\pi}{2n(n+1)}$ for even $n$.
		
		First, we assume $g=\sum_{j=1}^{2n+1}c_j\delta_{\alpha_j}$, with undetermined parameters $c_j\in\mathbb{C}$ and $\alpha_1<\alpha_2<\cdots<\alpha_{2n+1}<\alpha_1+2\pi$, where $\delta_{\alpha_j}$ is the Dirac measure satisfying $\int_{0}^{2\pi} f(\alpha)\delta_{\alpha_j}\dd\theta=f(\alpha_j)$ for all $f\in C(S^1)$. Moreover, without loss of generality, let $a_n = 1/\pi$ and by (\ref{vanishing_order}),
		\begin{eqnarray*}
			&&2\pi\widehat{g}(k) = \sum_{j=1}^{2n+1}c_je^{-ik\alpha_j} = 0, \quad\forall -(n-1)\leq k\leq n-1,\\
			&&2\pi\widehat{g}(n) = \sum_{j=1}^{2n+1}c_je^{-in\alpha_j} = \pi a_n = 1,\\
			&&2\pi\widehat{g}(-n) = \sum_{j=1}^{2n+1}c_je^{in\alpha_j} = \pi a_n = 1.
		\end{eqnarray*}
		Then we have the following non-homogeneous linear equations
		\begin{eqnarray}\label{35}
			\left(\begin{matrix}
				1 & 1 &\cdots & 1\\
				e^{-i\alpha_1} & e^{-i\alpha_2} & \cdots & e^{-i\alpha_{2n+1}} \\
				\vdots & \vdots & \vdots & \vdots \\
				e^{-i2n\alpha_1} & e^{-i2n\alpha_2} & \cdots & e^{-i2n\alpha_{2n+1}}
			\end{matrix}\right)
			\left(\begin{matrix}
				e^{in\alpha_1}c_1\\
				e^{in\alpha_2}c_2\\
				\vdots\\
				e^{in\alpha_{2n+1}}c_{2n+1}
			\end{matrix}\right)
			=
			\left(\begin{matrix}
				1\\
				0\\
				\vdots\\
				0\\
				1
			\end{matrix}\right),
		\end{eqnarray}
		which is equivalent to 
		\begin{eqnarray}\label{equi}
			\left(\begin{matrix}
				1 & 1 &\cdots & 1\\
				\sin\alpha_1 & \sin\alpha_2 & \cdots & \sin\alpha_{2n+1}\\
				\cos\alpha_1 & \cos\alpha_2 & \cdots & \cos\alpha_{2n+1}\\
				\vdots & \vdots & \vdots & \vdots \\
				\sin n\alpha_1 & \sin n\alpha_2 & \cdots & \sin n\alpha_{2n+1}\\
				\cos n\alpha_1 & \cos n\alpha_2 & \cdots & \cos n\alpha_{2n+1}\\
			\end{matrix}\right)
			\left(\begin{matrix}
				c_1\\
				c_2\\
				\vdots\\
				c_{2n+1}
			\end{matrix}\right)
			=
			\left(\begin{matrix}
				0\\
				0\\
				\vdots\\
				0\\
				1
			\end{matrix}\right).
		\end{eqnarray}
		This implies that $c_j\in\mathbb{R}$ for all $j$. 
		By the inverse of Vandermonde matrix, we have 
		\begin{eqnarray}\label{ccccccc}
			e^{in\alpha_j}c_j = \frac{1+\prod_{m=1,m\neq j}^{2n+1}e^{-i\alpha_m}}{\prod_{m=1,m\neq j}^{2n+1}(e^{-i\alpha_j}-e^{-i\alpha_m})}.
		\end{eqnarray}
		Then
		\begin{eqnarray}
			&&2\pi\widehat{g}(n+1) = \sum_{j=1}^{2n+1}c_je^{-i(n+1)\alpha_j} \nonumber\\
			&&= \sum_{j=1}^{2n+1}\frac{e^{-i(2n+1)\alpha_j}}{\prod_{m=1,m\neq j}^{2n+1}(e^{-i\alpha_j}-e^{-i\alpha_m})}
			+\sum_{j=1}^{2n+1}\frac{e^{-i(2n+1)\alpha_j}\prod_{m=1,m\neq j}^{2n+1}e^{-i\alpha_m}}{\prod_{m=1,m\neq j}^{2n+1}(e^{-i\alpha_j}-e^{-i\alpha_m})}\nonumber\\
			&&= \sum_{j=1}^{2n+1}\frac{e^{-i(2n+1)\alpha_j}}{\prod_{m=1,m\neq j}^{2n+1}(e^{-i\alpha_j}-e^{-i\alpha_m})}
			+\sum_{j=1}^{2n+1}\frac{e^{-i2n\alpha_j}\prod_{m=1}^{2n+1}e^{-i\alpha_m}}{\prod_{m=1,m\neq j}^{2n+1}(e^{-i\alpha_j}-e^{-i\alpha_m})}\label{338}.
		\end{eqnarray}
		Let $x_j:=e^{-i\alpha_j}$. By Lagrange interpolation, for $f(x)=x^{2n+1}$,
		\begin{eqnarray*}
			p_{2n+1}(x) := \sum_{j=0}^{2n+1}f(x_j)\prod_{m\neq j}^{1\leq m\leq {2n+1}} \frac{x-x_m}{x_j-x_m}.
		\end{eqnarray*}
		As $f(x)-p_{2n+1}(x)$ is a polynomial of degree $2n+1$ and has roots $x_j, 1\leq j\leq {2n+1}$, so by Vieta's formulas, 
		\begin{eqnarray}\label{339}
			\sum_{j=1}^{2n+1}\frac{x_j^{2n+1}}{\prod_{m=1,m\neq j}^{2n+1}(x_j-x_m)}=\sum_{j=1}^{2n+1}x_j.
		\end{eqnarray}
		Similarly, by Lagrange interpolation, for $f(x) = x^{2n}$,
		\begin{eqnarray*}
			\tilde{p}_{2n+1} := \sum_{j=0}^{2n+1}f(x_j)\prod_{m\neq j}^{1\leq m\leq {2n+1}} \frac{x-x_m}{x_j-x_m}.
		\end{eqnarray*}
		As $f(x)-\tilde{p}_{2n+1}(x)$ is a polynomial of degree $2n$ and has roots $x_j, 1\leq j\leq {2n+1}$, so $f\equiv \tilde{p}$, that is,
		\begin{eqnarray*}
			x^{2n}=\sum_{j=1}^{2n+1}x_j^{2n}\prod_{m\neq j}^{1\leq m\leq {2n+1}} \frac{x-x_m}{x_j-x_m}.
		\end{eqnarray*}
		Thus, 
		\begin{eqnarray}\label{3310}
			\sum_{j=1}^{2n+1}\frac{x_j^{2n}}{\prod_{m=1,m\neq j}^{2n+1}(x_j-x_m)}=1.
		\end{eqnarray}
		Substituting equation (\ref{339}) and (\ref{3310}) into equation (\ref{338}), we get
		\begin{eqnarray}
			2\pi\widehat{g}(n+1) &=&
			\sum_{j=1}^{2n+1}\frac{x_j^{2n+1}}{\prod_{m=1,m\neq j}^{2n+1}(x_j-x_m)}
			+\prod_{m=1}^{2n+1}x_m\sum_{j=1}^{2n+1}\frac{x_j^{2n}}{\prod_{m=1,m\neq j}^{2n+1}(x_j-x_m)}\nonumber
			\\
			&=& 
			\sum_{j=1}^{2n+1}x_j
			+\prod_{m=1}^{2n+1}x_m = \sum_{j=1}^{2n+1}e^{-i\alpha_j}+e^{-i\sum_{j=1}^{2n+1}\alpha_j}\label{kappa}.
		\end{eqnarray}
		
		When $n$ is odd, setting $q=(n-1)/2$ in (\ref{qqqqqq}) and noting $arg(w^{(n)}(0))=\widehat{g}(n)=0$, we have $(n+1)\eta_q = (n+1)\pi/2$. Combining (\ref{31000}) and letting $\alpha_j\rightarrow \alpha_0=0$ for all $j$, we obtain 
		\begin{eqnarray*}
			\av{k(u)\big|_{\gamma_q}(0)}&=&\av{\frac{2}{n}\frac{\RRe\left(e^{i(n+1)\eta_q}\widehat{g}(n+1)\right)}{\widehat{g}(n)}}\\
			&=& \av{\frac{2\left(\sum_{j=1}^{2n+1}\cos\alpha_j+\cos\sum_{j=1}^{2n+1}\alpha_j\right)}{n}}\\
			&\rightarrow& \frac{4(n+1)}{n}.
		\end{eqnarray*}
		
		When $n$ is even, we only need to adjust the value of $\alpha_j$ and $\eta_q$. Setting $q = n/2$ in (\ref{qqqqqq}), we have $(n+1)\eta_q=(n+1)^2\pi/(2n)$. Let $\alpha_j\rightarrow \alpha_0=\frac{\pi}{2n(n+1)}$ for all $j$. Substituting equation (\ref{kappa}) into (\ref{qqqqqq}), we derive
		\begin{eqnarray*}
			\av{k(u)\big|_{\gamma_q}(0)}&=&\av{\frac{2}{n}\frac{\RRe\left(e^{i(n+1)\eta_q}\widehat{g}(n+1)\right)}{\widehat{g}(n)}}\\
			&=& \av{\frac{2\RRe\left(e^{i\frac{(n+1)^2}{2n}\pi}(\sum_{j=1}^{2n+1}e^{-i\alpha_j}+e^{-i\sum_{j=1}^{2n+1}\alpha_j})\right)}{n}}\\
			&=& \av{\frac{2\RRe\left(e^{i\frac{\pi}{2n}}(\sum_{j=1}^{2n+1}e^{-i\alpha_j}+e^{-i\sum_{j=1}^{2n+1}\alpha_j})\right)}{n}}\\
			&\rightarrow& \av{\frac{2\RRe\left(e^{i\frac{\pi}{2n}}((2n+1)e^{-i\frac{1}{2n(n+1)}\pi}+e^{-i\frac{2n+1}{2n(n+1)}\pi})\right)}{n}}\\
			&=& \av{\frac{2\RRe\left((2n+1)e^{i\frac{\pi}{2(n+1)}}+e^{-i\frac{\pi}{2(n+1)}}\right)}{n}}\\
			&=&\frac{4(n+1)}{n}\cos\frac{\pi}{2(n+1)}.
		\end{eqnarray*}

		The lemma will be proven as long as we show that there is a sequence of harmonic functions $u_m(r,\theta)$ such that the curvature of nodal curves of $u_m(r,\theta)$ tends to the upper bound and $u_m(r,\theta)$ has only $n$ nodal curves in $B_1(0)$ intersecting at $0$. For the latter part, since the condition (\ref{vanishing_order}) means $n(u,0)=n$ and there are at least $n$ nodal curves in $B_1(0)$ for all $u|_{S^1}=g$, it is sufficient to show that $g_m=u_m|_{S^1}\in C^\infty(S^1)$ satisfying (\ref{vanishing_order}) and has only $2n$ zeros on $S^1$. 
		
		Recall $g=\sum_{j=1}^{2n+1}c_j\delta_{\alpha_j}, c_j\in\mathbb{R}$, where $\alpha_{j+1}-\alpha_{j}=\varepsilon_0\leq 1/(2n)$ for all $1\leq j\leq 2n$ and $\alpha_{n+1}=\alpha_0$. Equivalently, $\alpha_{j}=\alpha_0+(j-n-1)\varepsilon_0$ for $1\leq j\leq 2n$. Since $P(r,\cdot)$ is an approximate identity, we know that $u(r,\theta)=P(r,\cdot)\ast g(\cdot) (\theta)\in C^\infty(B_1(0))$ is harmonic. Let $\varphi(\theta)$ denote a non-negative smooth function with compact support $[-1,1]$, and let integral of $\varphi(\theta)$ on $[-1,1]$ be $1$. Set $\varphi_\varepsilon(\theta):=\varphi(\theta/\varepsilon)/\varepsilon$, and $\varepsilon<\varepsilon_0/4$. Thus $g_\varepsilon:=\varphi_\varepsilon\ast g$, where 
		\begin{eqnarray}\label{convolution}
			\varphi_\varepsilon\ast g(\theta) := 
			\int_{-\pi}^{\pi}\varphi_\varepsilon(\mu)g(\theta-\mu)\dd\mu.
		\end{eqnarray}
		Then we know $g_\varepsilon\in C^\infty(S^1)$ and $g_\varepsilon\rightarrow g$ in $(C^\infty(S^1))'$ as $\varepsilon\rightarrow 0+$. Furthermore, 
		\begin{eqnarray*}
			g_\varepsilon\ast P_r (\tau) &=& \int_{-\pi}^{\pi}\int_{-\pi}^{\pi}\varphi_\varepsilon(\mu)g(\theta-\mu)P_r (\tau-\theta)\dd\mu\dd\theta \\
			&=& \int_{-\pi}^{\pi}\int_{-\pi}^{\pi}\varphi_\varepsilon(\mu)g(\theta)P_r (\tau-\mu-\theta)\dd\theta\dd\mu\\
			&=& (g\ast P_r)_\varepsilon (\tau) =: u_\varepsilon(r,\tau)\in C^\infty(\overline{B_1(0)}).
		\end{eqnarray*}

		Therefore, $u_\varepsilon(r,\cdot)=g_\varepsilon\ast P_r\rightarrow g\ast P_r=u(r,\cdot)$ in $C^\infty(S^1)$ for $r<1$. 
		Moreover, $\widehat{g_\varepsilon}(k)=2\pi \widehat{\varphi_\varepsilon}(k)\widehat{g}(k)$, so by condition of $\widehat{g}$, we have
		\begin{eqnarray*}
			\widehat{g_\varepsilon}(k) = 0,\ \forall -(n-1)\leq k\leq n-1,\quad\ \widehat{g_\varepsilon}(n) \neq 0
		\end{eqnarray*}
		which means the vanishing degree of $u_\varepsilon$ at $0$ is $n$. Note that for $\varepsilon<\varepsilon_0/4$, $\forall\av{\theta-\alpha_j}<\varepsilon$, we have $g_\varepsilon(\theta) =  c_j\varphi_\varepsilon(\theta-\alpha_j)$. If $\av{\theta-\alpha_j}\geq\varepsilon$ for all $j$, then $g_\varepsilon(\theta)=0$. Now we need to make a small perturbation to the function $g_\varepsilon$ to construct a smooth function which has only $2n$ zeros on $S^1$.  
		Let $\beta_j:=(\alpha_j+\alpha_{j+1})/2$ for $1\leq j\leq 2n$, $\beta_{2n+1}:=\beta_1+2\pi$, and
		\begin{eqnarray}\label{simple_function}
			h(\theta) := \sum_{j=1}^{2n}d_j\mathbbm{1}_{[\beta_j,\beta_{j+1})}(\theta),
		\end{eqnarray}
		where $d_j$ is a constant for all $j$. Assume $h(\theta+2\pi)=h(\theta)$. 
		
		To determine every constant $d_j$, we assume $\widehat{h}(k) = 0, \forall 0\leq k\leq n-1; \Re\widehat{h}(n) = 1$. That is, 
		\begin{eqnarray}\label{inverse}
			\quad\left(\begin{matrix}
				\beta_2-\beta_1  &\cdots & \beta_{2n+1}-\beta_{2n}\\
				\sin\beta_2-\sin\beta_1  & \cdots & \sin\beta_{2n+1}-\sin\beta_{2n}\\
				\cos\beta_2-\cos\beta_1  & \cdots & \cos\beta_{2n+1}-\cos\beta_{2n}\\
				\vdots  & \vdots & \vdots \\
				\sin n\beta_2-\sin n\beta_1  & \cdots & \sin n\beta_{2n+1}-\sin n\beta_{2n}
			\end{matrix}\right)
			\left(\begin{matrix}
				d_1\\
				d_2\\
				\vdots\\
				d_{2n}
			\end{matrix}\right)
			=
			\left(\begin{matrix}
				0\\
				0\\
				\vdots\\
				0\\
				{2n\pi}
			\end{matrix}\right).
		\end{eqnarray}
		By Lemma \ref{auxilary}, such linear equation is solvable, which means $h\in L^\infty(S^1)$ satisfies (\ref{vanishing_order}). 
		
		Let $v(r,\theta):=P(r,\cdot)\ast h(\cdot)$, and $v_\varepsilon(r,\theta):=P(r,\cdot)\ast(\varphi_\varepsilon\ast h)=\varphi_\varepsilon(\cdot)\ast v(r,\cdot)\in C^\infty(\overline{B_1(0)})$. The boundary value $h_\varepsilon=\varphi_\varepsilon\ast h\in C^\infty(S^1)$ also satisfies (\ref{vanishing_order}) for $\varepsilon<\varepsilon_0/4$, which means $v$ has $n$ nodal curves intersecting at $0$, denoted by $\gamma_l, l=1,\ldots, n$. By Lemma \ref{l9}, we can assume the maximal curve $\gamma_l(0,1)\subset B_1(0)$ and $\gamma_l(0), \gamma_l(1)\in S^1$ for all $l$. 
		
		Let $I_j^a:=\{e^{i\theta}|\av{\theta-\beta_j}\leq \varepsilon\}$ and $I_j^b:=\{e^{i\theta}|\beta_{j}+\varepsilon<\theta<\beta_{j+1}-\varepsilon\}$ for $j=1,\ldots,2n$. Then $S^1=\cup_{j=1}^{2n}(I_j^a\cup I_j^b)$. Combining (\ref{convolution}) and (\ref{simple_function}), we find that for all $\theta\in I_j^b$, $h_\varepsilon(\theta)=d_j$ holds for all $1\leq j\leq 2n$. Moreover, if $d_jd_{j-1}<0$, then $h_\varepsilon$ is strictly monotonous in $I_j^a$; if $d_jd_{j-1}\geq 0$ and $d_j^2+d_{j-1}^2\neq 0$,  then $h_\varepsilon$ keeps positive or negative in $I_j^a$. If $d_j=d_{j-1}=0$, then $h_\varepsilon\equiv 0$ in $I_j^a$. So $h_\varepsilon$ has at most one zero in $I_j^a$ or vanishes in $I_j^a$. However, the set $\{v_\varepsilon=0\}\cap S^1$ has at least $2n$ points $\cup_{m=1}^{2n}\{\tau_m\}:=\{\gamma_l(0),\gamma_l(1)|l=1,\ldots, n\}$. Since $\gamma_l$ intersects at $0$, and by the maximum principle, there exists at most one point $\tau_{m}\in I_j^a$ for every $j$. 
		
		For convenience, we assume $d_0=d_{2n}$ and $d_1=d_{2n+1}$. For $1\leq j\leq 2n$ such that $d_j\neq 0$, $h_\varepsilon$ never vanishes in $I_j^b$. If there exists $j$ such that $d_j=d_{j-1}=0$, then by maximum principle, we have at most one point $\tau_m$ in $I_j^a\cup I_{j}^b\cup I_{j-1}^a\cup  I_{j-1}^b$. Similarly, if exists $j$ such that $d_{j-1}\neq 0, d_{j}=0$, then we have at most one point $\tau_m$ in $I_j^a\cup I_{j}^b\cup I_{j+1}^a\cup I_{j+1}^b$ (whether $d_{j+1}=0$ or $d_{j+1}\neq 0$). For all $j$ such that $d_jd_{j-1}<0$, then we have exactly one point $\tau_m$ in $I_j^a\cup I_{j}^b$. For all $j$ such that $d_jd_{j-1}>0$, then we have no point $\tau_m$ in $I_j^a\cup I_{j}^b$. Thus, letting $m_0:=\#\{d_jd_{j-1}= 0\ \text{or}\  d_jd_{j-1}>0 |j=1,\ldots, 2n\}$, we have at most $2n-m_0$ points of $\tau_m$, which is a contradiction when $m_0\geq 1$. So $m_0=0$. Thus, there does not exist $j$ such that $d_jd_{j-1}\geq 0$. That is, $d_jd_{j-1}<0$, and there is a unique point $\tau_m$ lying in $I_j^a$ for every $j$. By maximum principle, there does not exist other nodal curves in $B_1(0)$. That is, $h_\varepsilon$ has only $2n$ zeros on $S^1$. Without loss of generality, assume $(-1)^nd_{2n}>0$. 
		
		Recall $c_j$ is given by (\ref{ccccccc}), so

		\begin{eqnarray*}
			c_j &=& \frac{1+\prod_{m=1,m\neq j}^{2n+1}e^{-i\alpha_m}}{\prod_{m=1,m\neq j}^{2n+1}(e^{-i\alpha_j}-e^{-i\alpha_m})}e^{-in\alpha_j}\\
			&=&\frac{e^{-in(\alpha_0+(j-n-1)\varepsilon_0)}+e^{-i(3n\alpha_0+(n-1)(j-n-1)\varepsilon_0)}}{\prod_{m=1,m\neq j}^{2n+1}(e^{-i(\alpha_0+(j-n-1)\varepsilon_0)}-e^{-i(\alpha_0+(m-n-1)\varepsilon_0)})}\\
			&=&\frac{e^{i(n\alpha_0-n(j-n-1)\varepsilon_0)}+e^{i(-n\alpha_0-(n-1)(j-n-1)\varepsilon_0)}}{\prod_{m=1,m\neq j}^{2n+1}(e^{-i(j-n-1)\varepsilon_0}-e^{-i(m-n-1)\varepsilon_0})}\\
			&=&\frac{2\cos(n\alpha_0)+O(\varepsilon_0)}{\prod_{m=1,m\neq j}^{2n+1}(i(m-j)\varepsilon_0+O(\varepsilon_0^2))}\\
			&=&\frac{2\cos(n\alpha_0)+O(\varepsilon_0)}{(-1)^n\varepsilon_0^{2n}\prod_{m=1,m\neq j}^{2n+1}(m-j)+O(\varepsilon_0^{2n+1})}.
		\end{eqnarray*}
		
		Note that $\cos(n\alpha_0)>0$ for all $n$, whether $\alpha_0=0$ or $\alpha_0=\frac{\pi}{2n(n+1)}$. Hence, the sign of $c_j$ is determined by $(-1)^n\prod_{m=1,m\neq j}^{2n+1}(m-j)$. That is, $c_jc_{j-1}<0$ for all $2\leq j\leq n+1$ and $(-1)^{n}c_1>0,(-1)^{n}c_{2n+1}>0$. Therefore, for all $\theta$ such that $g_\varepsilon(\theta)>0$, we have $h_\varepsilon(\theta)>0$; for all $\theta$ such that $g_\varepsilon(\theta)<0$, we have $h_\varepsilon(\theta)<0$. Let $\lambda>0$ and
		
		$$u_{\varepsilon,\lambda}:=u_\varepsilon+\lambda v_\varepsilon.$$
		Note that adding $g_\varepsilon$ does not change the sign of $h_\varepsilon$ and does not change the zeros of $h_\varepsilon$, which means $g_{\varepsilon,\lambda}:=g_\varepsilon+\lambda h_\varepsilon$ has only $2n$ zeros on $S^1$. Thus, $u_{\varepsilon,\lambda}$ has exactly $n$ nodal curves in $B_1(0)$, for all $\varepsilon<\varepsilon_0/4$. Moreover, 
		\begin{eqnarray*}
			2\pi\widehat{\varphi_\varepsilon}(k) = \int_{-\pi}^{\pi}\varphi_\varepsilon(\theta)e^{-ik\theta}\dd\theta\rightarrow 1, \ as\ \varepsilon\rightarrow 0+ .
		\end{eqnarray*}
		Hence, letting $\lambda\rightarrow0+$ and then $\varepsilon\rightarrow 0+$, we have
		\begin{eqnarray*}
			\av{k(u_{\varepsilon,\lambda})\big|_{\gamma_q}(0)}&=&\av{\frac{2}{n}\frac{\RRe\left(e^{i(n+1)\eta_q}(\widehat{g_\varepsilon}(n+1)+\lambda\widehat{h_\varepsilon}(n+1))\right)}{(\widehat{g_\varepsilon}(n)+\lambda\widehat{h_\varepsilon}(n))}}\\
			&=& \av{\frac{2}{n}\frac{\RRe\left(e^{i(n+1)\eta_q}\widehat{\varphi_\varepsilon}(n+1)(\widehat{g}(n+1)+\lambda\widehat{h}(n+1))\right)}{\widehat{\varphi_\varepsilon}(n)(\widehat{g}(n)+\lambda\widehat{h}(n))}}\\
			&\rightarrow& \av{k(u)\big|_{\gamma_q}(0)},
		\end{eqnarray*}
		as $\alpha_j\rightarrow \alpha_0$. Then we get a sequence of harmonic functions $u_m$, which has only $n$ nodal curves intersecting at $0$. Moreover, $\av{\kappa(u_m)\big|_{\gamma_q}(0)}$ tends to the upper bound.
	\end{proof}
	Now we are going to show the existence and the uniqueness of the extremer. ~\\
	
	\emph{Proof of Theorem \ref{extremer_case}.}
	For $T\in(C^\infty(S^1))^\prime$, the Poisson extension of $T$ is given by $u(r,\theta):=T(P_r(\theta-\cdot))$. Note $u(0)=T({1}/{2\pi})=0$. Extend $T$ to a complex distribution by $ T(f_1+if_2):=T(f_1)+iT(f_2), \forall f_1,f_2\in C^\infty(S^1)$ and note $P_r(\theta-\alpha)=\frac{1}{2\pi}\RRe\left(\frac{\zeta+z}{\zeta-z}\right)$, where $z=re^{i\theta},\zeta=e^{i\alpha}$. Therefore, for $r<1$, we have 
	\begin{eqnarray*}
		u(z) = \RRe \frac{1}{2\pi}T\left(\frac{\zeta+z}{\zeta-z}\right) 
		= \RRe \frac{1}{2\pi}T\left(\frac{2z}{\zeta-z}\right) 
		\in C^\infty(B_1(0)).
	\end{eqnarray*}
	Moreover, the above equality implies $u$ is the real part of some holomorphic function. Thus, $u$ is indeed a harmonic function in $B_1(0)$. Observing that $w_{k_0}(\zeta):=\sum_{k=k_0}^{+\infty}\zeta^{-k}z^k\xrightarrow{C^{\infty}(S^1)}0$ as $k_0\rightarrow \infty$, we get
	\begin{eqnarray*}
		u(z) = \RRe \frac{1}{\pi}T\left(\frac{\zeta^{-1}z}{1-\zeta^{-1}z}\right)
		=\RRe \frac{1}{\pi}T\left(\sum_{k=1}^{+\infty}\zeta^{-k}z^k\right)
		=\RRe \sum_{k=1}^{+\infty}\frac{1}{\pi}T\left(\zeta^{-k}\right)z^k,
	\end{eqnarray*}
	where $T\left(\zeta^{-k}\right)/2\pi$ can be regarded as $\widehat{g}(k)$ in the proof of Lemma \ref{L1}. Since the nodal set of $u$ at $0$ is the union of $n$ nodal curves, we have $T\left(\zeta^{-n}\right)\neq 0, T\left(\zeta^{-k}\right)=0, \forall 0\leq k\leq n-1$. Denoting $a_ke^{i\theta_k}=T\left(\zeta^{-k}\right)/{\pi}$ and $u=\RRe  w$, we still have the expansion
	$$
	w(z)=a_nz^n+a_{n+1}e^{i\theta_{n+1}}z^{n+1}+\cdots.
	$$
	Regarding integral as operation of the distribution $T$ which almost changes its sign for
	$2n$ times, with the same proof in Theorem \ref{Main_theorem}, we derive the same curvature estimates (\ref{e1}) and (\ref{e2}) for $u$.
	
	Now we give the conditions of $\varphi_j$ and $\eta_q$ when equality holds in Theorem \ref{Main_theorem}. We may assume $\varphi_j\in [0,2\pi)$ for all $j=0,1,\ldots,n$. 
	
	For odd $n$, we find that $\varphi_j$ has to be the same by equations (\ref{33}) and (\ref{onlyif1}), that is $\varphi_{j}=\varphi_0$ for all $j$. Moreover, with $\overline{\varphi}=\sum_{j=1}^{n}\varphi_j$, we have
	\begin{eqnarray}\label{444444}
		(n+1)\varphi_0\equiv(n+1)\eta_q\ (mod\ \pi).
	\end{eqnarray}
	
	For even $n$, note that equations (\ref{32}) and (\ref{33}) can not hold simultaneously, which implies that for equality case in (\ref{onlyifeven}), $e^{i(\theta_{n+1}+(n+1)\eta_q)}\neq e^{i(-\varphi_j+(n+1)\eta_q)}$ for some $j$. From conditions of equality case in the first three inequalities in (\ref{341}), we know that $c_j=-1$ for all $j=1,\ldots, n$ and
	\begin{eqnarray}\label{420}
		\overline{\varphi}\equiv-\varphi_j+(n+1)\eta_q+l\pi\ (mod\ 2\pi),\ \forall 1\leq j \leq n,
	\end{eqnarray}
	where $l\in\{0,1\}$ satisfies $\cos(-\varphi_j+(n+1)\eta_q+l\pi)>0$.
	Note that 
	\begin{eqnarray*}
		&&{a_{n+1}e^{i(\theta_{n+1}+(n+1)\eta_q)}+2\sum_{j=1}^nc_ja_ne^{i(-\varphi_j+(n+1)\eta_q)}}\\
		&&=\left({a_{n+1}e^{i(\theta_{n+1}+\bar{\varphi}+\varphi_0)}+2\sum_{j=1}^nc_ja_ne^{i(\bar{\varphi}+\varphi_0-\varphi_j)}}\right)e^{i(-\bar{\varphi}-\varphi_0+(n+1)\eta_q)}.
	\end{eqnarray*}
	Combining (\ref{30}) and (\ref{onlyifeven}), we know that 
	\begin{eqnarray}\label{421}
		(n+1)\eta_q-\varphi_{0}-\overline{\varphi}\equiv l\pi\ (mod\ 2\pi).
	\end{eqnarray}
	From (\ref{420}) and (\ref{421}), we know that 
	\begin{eqnarray}\label{422}
		\varphi_{j}\equiv \varphi_{0}\ (mod\ 2\pi),\ \forall 1\leq j \leq n.
	\end{eqnarray}
	Moreover, (\ref{444444}) still holds.

	Therefore, for all $n$, equation (\ref{422}) always holds. Without loss of generality, we may assume $\varphi_{j}=\varphi_0$ for all $j$.
	Then $g_0$ in Definition \ref{def1} can be given by $g_0(\theta)=(1 - \cos (\theta-\varphi_0))^n$. In this case, the inequality (\ref{integrate}) has to be an equality for $k=1$, which means for $\varphi_0$,
	\begin{eqnarray*}
		T\left((1 - \cos(\theta - \varphi_0))g_0(\theta)\right) = 0.
	\end{eqnarray*}
	But we know $T$ almost changes its sign for $2n$ times, so from (\ref{change_sign}) we claim for all non-negative functions $f\in C^\infty(S^1)$ with $f(\varphi_0)=0$,
	\begin{eqnarray}\label{distribution}
		T\left(g_0(\theta)f(\theta)\right)=0.
	\end{eqnarray}
	If not, we may assume that there is some smooth function $f\geq 0$, such that $T\left(g_0(\theta)f(\theta)\right)>0$. Note that ${f}$ is non-negative and ${f}(\varphi_0)=0, {f}^\prime(\varphi_0)=0$, which means ${f}$ is convex at $\varphi_0$. Therefore, there exists some small constant $\varepsilon<1$, such that $(1 - \cos(\theta - \varphi_0)) - \varepsilon{f}$ is non-negative on $S^1$. However, this means 
	\begin{eqnarray*}
		T\left(g_0(\theta)((1 - \cos(\theta - \varphi_0)) - \varepsilon{f})\right)=-\varepsilon T(g_0(\theta)f(\theta))<0,
	\end{eqnarray*}
	which is contradictory to (\ref{change_sign}). 
	
	Thus, (\ref{distribution}) holds for all $f\in C^\infty(S^1)$ with $f\geq 0, f(\varphi_0)=0$.
	Since $g_0T$ is a positive measure, we get that $g_0T$ is supported on $\{\varphi_0\}$ and hence $T$ is a distribution supported on a single point $\{\varphi_0\}$. Thus, $T$ is a finite linear combination of $\delta_{\varphi_0}$ and its derivatives. Assume $T=\sum_{j=0}^{m}b_jD^j\delta_{\varphi_0}$ and note $g_0(\theta)=(1 - \cos (\theta-\varphi_0))^n$. For all $f\geq 0, f(\varphi_0)=0$, we have
	\begin{eqnarray*}
		T\left((1 - \cos(\theta - \varphi_0))^nf(\theta)\right) = 0.
	\end{eqnarray*}
	
	If $m = 2n+2q$ for $q\geq 1$, then we have
	\begin{eqnarray*}
		T\left((1 - \cos(\theta-\varphi_0))(1 - \cos (\theta-\varphi_0))^{q-1}(1 - \cos (\theta-\varphi_0))^n\right) = c_mb_m = 0,
	\end{eqnarray*}
	where $c_m$ is a non-zero constant. Hence $b_m=0$.
	
	If $m = 2n+2q-1$ for $q\geq 1$, then we have
	\begin{eqnarray*}
		T\left(\sin(\theta-\varphi_0)(1 - \cos (\theta-\varphi_0))^{q-1}(1 - \cos (\theta-\varphi_0))^n\right) = c_mb_m = 0,
	\end{eqnarray*}
	hence $b_m=0$.
	Therefore, $T=\sum_{j=0}^{2n}b_jD^j\delta_{\varphi_0}$. Note $T(1)=b_0=0$. Moreover, from the condition that $n$ nodal curves intersect at $0$, after normalization, we may assume
	\begin{eqnarray*}
		&&\frac{1}{\pi}T\left(\zeta^{-k}\right) = \frac{1}{\pi}\sum_{j=1}^{2n}b_j(ik)^je^{-ik\varphi_0} = 0, \quad\forall 1\leq k\leq n-1,\\
		&&\frac{1}{\pi}T\left(\zeta^{-n}\right) = \frac{1}{\pi}\sum_{j=1}^{2n}b_j(in)^je^{-in\varphi_0} = 1.
	\end{eqnarray*}
	Calculating the real part and the imaginary part respectively, we obtain
	\begin{eqnarray}\label{aaaaaaa}
		\left(\begin{matrix}
			1 & 1 &\cdots & 1\\
			2^2 & 2^4 & \cdots & 2^{2n}\\
			\vdots & \vdots & \vdots & \vdots \\
			n^2 & n^4 & \cdots & n^{2n}\\
		\end{matrix}\right)
		\left(\begin{matrix}
			-b_2\\
			b_4\\
			\vdots\\
			(-1)^nb_{2n}
		\end{matrix}\right)
		=
		\left(\begin{matrix}
			0\\
			\vdots\\
			0\\
			\pi\cos n\varphi_0
		\end{matrix}\right),
	\end{eqnarray}
	and 
	\begin{eqnarray}\label{bbbbbbb}
		\left(\begin{matrix}
			1 & 1 &\cdots & 1\\
			2^1 & 2^3 & \cdots & 2^{2n-1}\\
			\vdots & \vdots & \vdots & \vdots \\
			n^1 & n^3 & \cdots & n^{2n-1}\\
		\end{matrix}\right)
		\left(\begin{matrix}
			b_1\\
			-b_3\\
			\vdots\\
			(-1)^{n-1}b_{2n-1}
		\end{matrix}\right)
		=
		\left(\begin{matrix}
			0\\
			\vdots\\
			0\\
			\pi\sin n\varphi_0
		\end{matrix}\right),
	\end{eqnarray}
	which gives the formula of $T=\sum_{j=1}^{2n}b_jD^j\delta_{\varphi_0}$. Set $\widetilde{w}(z)={w}(e^{i\varphi_0}z)$ and corresponding boundary distribution $\widetilde{T}(z)={T}(e^{i\varphi_0}z)=\sum_{j=1}^{2n}b_jD^j\delta_{0}$. Thus, by expansion, we have
	\begin{eqnarray*}
		\widetilde{w}(z)&=&\frac{1}{\pi}\sum_{k=1}^{+\infty}\widetilde{T}\left(e^{-ik\theta}\right)z^k\\
		&=& \frac{1}{\pi}\sum_{k=1}^{+\infty}\left(\sum_{j=1}^{2n}(-1)^j(-ik)^jb_j\right)z^k\\
		&=& \frac{1}{\pi}\sum_{k=n}^{+\infty}\left(\sum_{j=1}^{n}(-1)^jk^{2j}b_{2j}+i\sum_{j=1}^{n}(-1)^{j-1}k^{2j-1}b_{2j-1}\right)z^k.
	\end{eqnarray*}
	With the inverse of Vandermonde matrix, we derive
	\begin{eqnarray*}
	\frac{(-1)^jb_{2j}}{\pi} = \frac{(-1)^{n-j}e_{n-j}(\{1^2,2^2,\ldots,(n-1)^2\})}{n^2\prod_{m=1}^{n-1}(n^2-m^2)}\cos n\varphi_0,\ \forall 1\leq j\leq n,
	\end{eqnarray*}	
	where $e_{n-j}$ denotes the elementary symmetric polynomial of degree $n-j$ and $e_0=1$. Therefore, setting $C=1/n^2\prod_{m=1}^{n-1}(n^2-m^2)=1/n(2n-1)!$, we have
	\begin{eqnarray*}
		\frac{1}{\pi}\sum_{j=1}^{n}(-1)^jk^{2j}b_{2j} &=& C(-1)^{n-j}e_{n-j}(\{1^2,2^2,\ldots,(n-1)^2\})k^{2j}\cos n\varphi_0\\
		&=& Ck^2\prod_{j=1}^{n-1}(k^2-j^2)\cos n\varphi_0.
	\end{eqnarray*}	
	Note that $k\prod_{j=1}^{n-1}(k^2-j^2)=\left(\begin{matrix}
		k+n-1\\ k-n
	\end{matrix}\right)(2n-1)!$. We obtain 
	\begin{eqnarray*}
		\frac{1}{\pi}\sum_{k=n}^{+\infty}\sum_{j=1}^{n}(-1)^jk^{2j}b_{2j}z^k
		&=& \sum_{k=n}^{+\infty}Ck\left(\begin{matrix}
			k+n-1\\ k-n
		\end{matrix}\right)(2n-1)!z^k\cos n\varphi_0\\
		&=& Cz^n \sum_{k=0}^{+\infty}(n+k)\left(\begin{matrix}
			k+2n-1\\ k
		\end{matrix}\right)(2n-1)!z^k\cos n\varphi_0\\
		&=& C(2n-1)!z^n \left(\frac{n}{(1-z)^{2n}}+\frac{2nz}{(1-z)^{2n+1}}\right)\cos n\varphi_0.
	\end{eqnarray*}	
	Similarly, we have
	\begin{eqnarray*}
		\frac{1}{\pi}\sum_{j=1}^{n}(-1)^{j-1}k^{2j-1}b_{2j-1} = Cnk\prod_{j=1}^{n-1}(k^2-j^2)\sin n\varphi_0,
	\end{eqnarray*}	
	and 
	\begin{eqnarray*}
		\frac{1}{\pi}\sum_{k=n}^{+\infty}\sum_{j=1}^{n}(-1)^{j-1}k^{2j-1}b_{2j-1}z^k
		&=& \sum_{k=n}^{+\infty}Cn\left(\begin{matrix}
			k+n-1\\ k-n
		\end{matrix}\right)(2n-1)!z^k\sin n\varphi_0\\
		&=& C(2n-1)!\frac{nz^n }{(1-z)^{2n}}\sin n\varphi_0
	\end{eqnarray*}	
	In conclusion, noting $Cn(2n-1)!=1$, we get
	\begin{eqnarray*}
		{\widetilde{w}(z)}=\frac{z^n }{(1-z)^{2n}}e^{in\varphi_0}+\frac{2z^{n+1}}{(1-z)^{2n+1}}\cos n\varphi_0.
	\end{eqnarray*}	
	Thus, 
	\begin{eqnarray}\label{extremer}
		{w}(z)=\frac{z^n }{(1-e^{-i\varphi_0}z)^{2n}}+\frac{2z^{n+1}}{(1-e^{-i\varphi_0}z)^{2n+1}}e^{-i(n+1)\varphi_0}\cos n\varphi_0,
	\end{eqnarray}	
	which implies ${w}^{(n)}(0)=n!$ and 
	\begin{eqnarray*}
		{w}^{(n+1)}(0)=2n(n+1)!e^{-i\varphi_{0}}+2(n+1)!e^{-i(n+1)\varphi_{0}}\cos n\varphi_0.
	\end{eqnarray*}	
	To calculate the value of $\varphi_0$, we use equation (\ref{000}) and derive that
	\begin{eqnarray}
		\nonumber\kappa(u)\big|_{\gamma_q}(0)&=&  -\frac{2}{n^2+n}\frac{\RRe\left(e^{i(n+1)\eta_q}w^{(n+1)}(0)\right)}{\av{w^{(n)}(0)}}.\\
		&=&-\frac{4}{n}\RRe\left(e^{i(n+1)\eta_q}(ne^{-i\varphi_{0}}+e^{-i(n+1)\varphi_{0}}\cos n\varphi_0)\right).\label{unique}
	\end{eqnarray}	
	
	For odd $n$, the extreme case $\av{k(u)\big|_{\gamma_q}}=4(n+1)/n$. Hence, (\ref{unique}) tells us 
	\begin{eqnarray*}
		\av{\RRe\left(e^{i(n+1)\eta_q}(ne^{-i\varphi_{0}}+e^{-i(n+1)\varphi_{0}}\cos n\varphi_0)\right)} = n+1.
	\end{eqnarray*}
	Recall $\eta_q=\frac{q\pi+\frac{1}{2}\pi}{n}$. As a result, with equation (\ref{444444}), there are some integers $k_1$ and $k_2$, such that
	\begin{eqnarray*}
		\varphi_0=\frac{k_1\pi}{n},\ (n+1)(q+\frac{1}{2})-(n+1)k_1=nk_2.
	\end{eqnarray*}
	Therefore, letting $k_3=k_2+k_1-q=1$, we have
	\begin{eqnarray*}
		q+\frac{n+1}{2}-k_1=n.
	\end{eqnarray*}
	For all $0\leq k_1\leq n-1$, there exists a $0\leq q\leq 2n-1$ satisfying above equation. Hence, for all $0\leq k_1\leq n-1$, $u(z)=\RRe w(z)$ in (\ref{extremer}) with $\varphi_0=\frac{k_1\pi}{n}$ gives a unique extremer of inequality (\ref{e1}) (note that different $k_1$ corresponds to different nodal curve which attains curvature upper bound at $0$).
	
	For even $n$, the extreme case gives $\av{k(u)\big|_{\gamma_q}}=4(n+1)\cos(\frac{\pi}{2(n+1)})/n$. With equation (\ref{444444}), we know there is some integer $k_2$ such that $(n+1)\eta_q-(n+1)\varphi_0=k_2\pi$. Then from (\ref{unique}), we have 
	\begin{eqnarray*}
		\av{\RRe\left(e^{i(n+1)\eta_q}(ne^{-i\varphi_{0}}+e^{-i(n+1)\varphi_{0}}\cos n\varphi_0)\right)} = (n+1)\cos(n\varphi_0).
	\end{eqnarray*}
	Thus, $\av{\cos(n\varphi_0)}=\cos(\frac{\pi}{2(n+1)})$, which means there is some integer $k_1$, such that
	\begin{eqnarray*}
		\varphi_0=\frac{k_1\pi}{n}+\frac{\pi}{2n(n+1)},\ (n+1)(q+\frac{1}{2})-(n+1)(k_1+\frac{1}{2(n+1)})=nk_2.
	\end{eqnarray*}
	Therefore, letting $k_3=k_2+k_1-q=1$, we have
	\begin{eqnarray*}
		q+\frac{n}{2}-k_1=n.
	\end{eqnarray*}
	Similarly, for all $0\leq k_1\leq n-1$, there exists a $0\leq q\leq 2n-1$ satisfying above equation. Hence, for all $0\leq k_1\leq n-1$, $u(z)=\RRe w(z)$ in (\ref{extremer}) with $\varphi_0=\frac{k_1\pi}{n}+\frac{\pi}{2n(n+1)}$ gives a unique extremer of inequality (\ref{e2}).
	
	To end the proof, we need to show the Poisson extension $u$ of  $g(\theta)$ has only $n$ nodal curves in $B_1(0)$, intersecting at $0$. Note that the holomorphic function $w$ is given by
	\begin{eqnarray*}
		{w}(z)&=&\frac{z^n }{(1-e^{-i\varphi_0}z)^{2n}}+\frac{2z^{n+1}}{(1-e^{-i\varphi_0}z)^{2n+1}}e^{-i(n+1)\varphi_0}\cos n\varphi_0\\
		&=& \frac{2h(z)\cos n\varphi_0}{(1-e^{-i\varphi_0}z)^{2n+1}},
	\end{eqnarray*}
	where $h(z)$ has no singular points and $h(z)\rightarrow 1$ as $z\rightarrow e^{i\varphi_{0}}$. 
	Hence there are some $\varepsilon_0>0$ such that for all $\av{e^{i\varphi_{0}}-z}<\varepsilon_0$, $\sqrt[2n+1]{h(z)}$ is a single valued holomorphic function. Let
	\begin{eqnarray*}
		\frac{1}{Z} := \frac{\sqrt[2n+1]{h(z)}}{1-e^{-i\varphi_{0}}z},
	\end{eqnarray*}
	which gives a biholomorphic map $Z:B_{\varepsilon_0}(e^{i\varphi_{0}})\rightarrow Z(B_{\varepsilon_0}(e^{i\varphi_{0}}))\subset B_{2\varepsilon_0}(0)$. Thus
	\begin{eqnarray*}
		u(z) = \RRe w(z) = \RRe \frac{2\cos n\varphi_0}{Z^{2n+1}}.
	\end{eqnarray*}
	Then $u(z)=0$ is equivalent to 
	$$\RRe \frac{1}{Z^{2n+1}}=\RRe \frac{\bar{Z}^{2n+1}}{\av{Z}^{4n+2}}=\RRe \frac{{Z}^{2n+1}}{\av{Z}^{4n+2}}=0.$$
	Following from that $\RRe{{Z}^{2n+1}}=0$ gives $2n+1$ $Z$-nodal curves and $Z$ is a biholomorphic map, we have $u(z)=0$ has $2n+1$ $z$-nodal curves in $B_{\varepsilon_0}(e^{i\varphi_{0}})\backslash\{e^{i\varphi_{0}}\}$.
	
	Since $w(z)$ has only one singular point at $e^{i\varphi_{0}}$, we know that $u(z)=\RRe w(z)$ is well defined and harmonic in $\mathbb{C}\backslash\{e^{i\varphi_{0}}\}$.
	Furthermore, boundary value $T$ is supported on $\{e^{i\varphi_{0}}\}$, so $u|_{S^1\backslash\{e^{i\varphi_{0}}\}}=0$.
	
	If there is some nodal curve $\gamma$ in $B_1(0)$ whose end $z_0$ is on $S^1\backslash\{e^{i\varphi_{0}}\}$, then $\gamma$ intersects $S^1\backslash\{e^{i\varphi_{0}}\}$ at $z_0$, and we know that there are at least two nodal curves across $z_0$. Thus, $\av{w^\prime(z_0)}=\av{\nabla u(z_0)}=0$, which is equivalent to $\av{\widetilde{w}^\prime(e^{-i\varphi_{0}}z_0)}=0$. But by direct calculation,
	\begin{eqnarray*}
		\widetilde{w}^\prime(z) = \frac{nz^{n-1}(1+z) }{(1-z)^{2n+1}}e^{in\varphi_0}+\frac{\left((n+1)z^{n}+nz^{n+1}\right)2\cos n\varphi_0}{(1-z)^{2n+2}},
	\end{eqnarray*}
	which means $w^\prime(z_0)=0$ if and only if for $z=e^{-i\varphi_{0}}z_0:=e^{i\theta}\in S^1\backslash\{1\}$,
	\begin{eqnarray}
		&&n(1-z^2)e^{in\varphi_0}+\left((n+1)z+nz^{2}\right)2\cos n\varphi_0 \nonumber\\&&= ne^{in\varphi_0}+nz^2e^{-in\varphi_0}+(n+1)(e^{in\varphi_0}+e^{-in\varphi_0})z\nonumber\\
		&&=z\left(n\overline{z}e^{in\varphi_0}+nze^{-in\varphi_0}+(n+1)(e^{in\varphi_0}+e^{-in\varphi_0})\right)\nonumber\\
		&&=2z\left(n\cos{(\theta-n\varphi_0)}+(n+1)\cos(n\varphi_0)\right)=0,\label{derivative}
	\end{eqnarray}
	where $\av{\cos(n\varphi_0)}=1$ for odd $n$ and $\av{\cos(n\varphi_0)}=\cos\frac{\pi}{2(n+1)}$ for even $n$. For odd $n$, equation (\ref{derivative}) can not hold for any $\theta$. For even $n$, set $f(x)=\cos(\pi x/2)-(1-x)$ for $0<x\leq 1/3$. Then $f^{\prime\prime}(x)=-\pi^2/4\cos(\pi x/2)<0$. So $f$ is concave on $0<x\leq 1/3$. Note that $f(1/3)>f(0)$. Therefore, $f(x)>f(0)=0$ for all $0<x\leq 1/3$, and then $\cos(\pi /2(n+1))>n/(n+1)$, which means equation (\ref{derivative}) can not hold for any $\theta$. In conclusion, each end of every nodal curve in $B_1(0)$ must be $e^{i\varphi_{0}}$.
	
	Now assume $u$ has at least $n+1$ nodal curves $\cup_{j=1}^{n+1}\gamma_j$, $\gamma_1,\ldots,\gamma_n$ intersect at $0$ and $\gamma_{n+1}$ does not pass $0$. By  Lemma \ref{l9}, for all $j$, both ends of $\gamma_j$ in $B_{1}(0)$ will extend to the boundary. From above discussion, we know that each end of every nodal curve is $e^{i\varphi_{0}}$ . Therefore, there are at least $2n+2$ ends coincident with $e^{i\varphi_{0}}$. Since $\av{z}=1$ is also a $z$-nodal curve and each end represents a nodal curve in $B_{\varepsilon}(1)$ for small $\varepsilon>0$, we have at least $2n+3$ $z$-nodal curves in $B_{\varepsilon}(1)$, which gives a contradiction. In conclusion, $u$ has at most $n$ nodal curves in $B_{1}(0)$. But initial data (\ref{aaaaaaa}) implies that $u$ has $n$ nodal curves intersecting at $0$. So we derive that $u$ has exactly $n$ nodal curves in $B_{1}(0)$ and intersecting at point $0$ and point $e^{i\varphi_{0}}$.
	\hfill$\square$\par
	\begin{remark}
		The assumptions that $u$ is a Poisson extension of a distribution $T$, $T$ almost changes its sign for $2n$ times and $u$ has $n$ nodal curves intersecting at $0$ can not be removed. By the result of \cite{ramm2022dirichlet} and Lemma \ref{l9}, the condition $T=g\in L^1(S^1)$ has only $n$ nodal curves intersecting at $0$ is enough to show $u$ is a Poisson extension of $g$ and $g$ almost changes its sign for $2n$ times. But for a distribution $T$, the uniqueness of harmonic extension $u$ of boundary value $T$ is unknown and  the property of nodal curves of $Z(u)$ is unknown near $S^1$. 
	\end{remark}
	\begin{figure}[htbp]
		\centering  
		\subfigure[$u=0$]{
			\includegraphics[width=4.2cm]{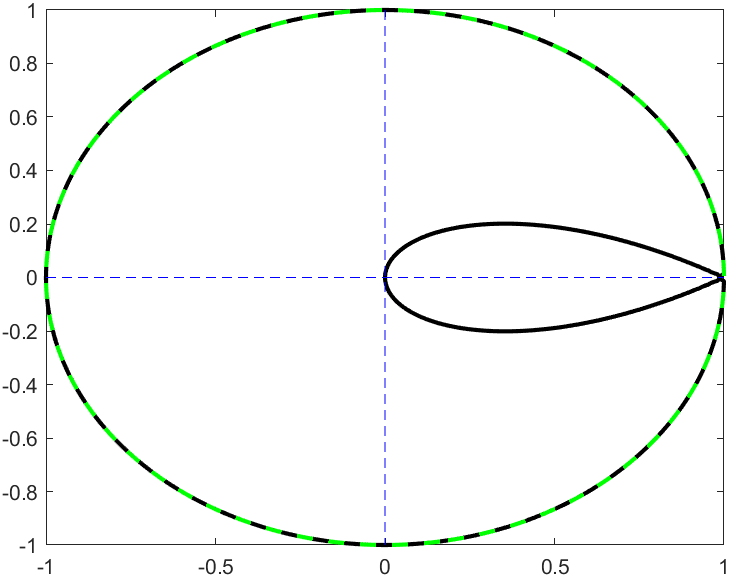}}
		\subfigure[$\tilde{u}=0$]{
			\includegraphics[width=4.2cm]{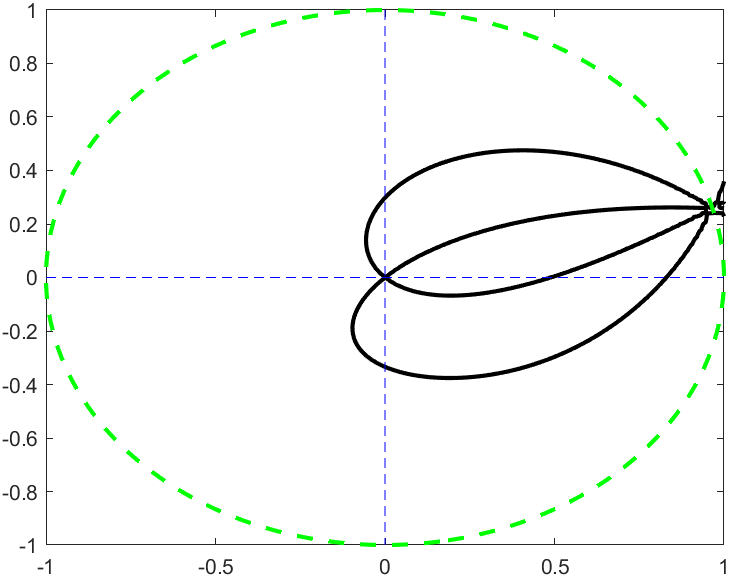}}
		\caption{Nodal sets of $u$ and $\tilde{u}$}
	\end{figure}
	
	For $n=1$, $T=D^2\delta_0$, and the corresponding harmonic function $u(x,y)$ is formulated by
	\begin{eqnarray*}
		\frac{2\left(1-x^2-y^2\right)\left(x(1-2x+x^2+y^2)+4y^2\right)}{\left(x^2- 2x+y^2+1\right)^3}.
	\end{eqnarray*}
	
	For $n=2$, $\widetilde{T}=-\frac{1}{12}D^1\delta_{\frac{\pi}{12}}+\frac{1}{8\sqrt{3}}D^2\delta_{\frac{\pi}{12}}-\frac{1}{12}D^3\delta_{\frac{\pi}{12}}/12+\frac{1}{8\sqrt{3}}D^4\delta_{\frac{\pi}{12}}$, and $\widetilde{u}(x,y)$ denotes the corresponding harmonic function.
	
	The graph of $\{u=0\}$ and $\{\tilde{u}=0\}$ is illustrated as above. We can see that $Z(u)$ almost forms a closed contour in $B_1(0)$, and so does every nodal curve of $Z(\tilde{u})$.

	\section{Curvature estimates at zeros around $0$ and area estimates}
	Without loss of generality,  we take $r_0=1$ in Theorem \ref{theorem2} , where $u$ has only $n\geq 1$ nodal curves intersecting at the origin in $B_1(0)$. Combining Taylor expansion in Lemma \ref{L3.12} and growth estimates in Lemma \ref{L1}, we can give a proof of Theorem \ref{theorem2}.~\\
	
	\emph{Proof of Theorem \ref{theorem2}}.
	Note that $\forall z\in Z(u)$, $\nabla u(z)\neq 0$. Thus, equation (\ref{E3.3}) holds.
	By the formula $w(z)=\frac{1}{\pi i}\int_{S^1} \frac{g(\zeta)}{\zeta-z}\dd\zeta$, we have
	\begin{eqnarray*}
		w^\prime(z) &=& \frac{1}{\pi} \int_{0}^{2\pi} g(\theta)e^{-i\theta}(1-ze^{-i\theta})^{-2}\dd\theta = z^{n-1}\sum_{k=n}^\infty ka_{k}e^{i\theta_k}z^{k-n}\\
		w^{\prime\prime}(z) &=& \frac{2}{\pi} \int_{0}^{2\pi} g(\theta)e^{-i2\theta}(1-ze^{-i\theta})^{-3}\dd\theta = z^{n-2}\sum_{k=n}^\infty k(k-1)a_{k}e^{i\theta_k}z^{k-n}.
	\end{eqnarray*}
	With above equations and equation (\ref{E3.3}), we can get
	\begin{eqnarray*}
		\av{\kappa(u)(z)}&=&\av{\frac{\RRe\left(w^{\prime\prime}(z)\overline{w^{\prime}(z)}^2\right)}{w^{\prime}(z)^3}}\\
		&=&\av{\frac{\RRe\left(({z}^{n-2}\sum_{k=n}^\infty k(k-1)a_{k}e^{i\theta_k}{z}^{k-n})(\bar{z}^{n-1}\sum_{k=n}^\infty ka_{k}e^{-i\theta_k}\bar{z}^{k-n})^2\right)}{(z^{n-1}\sum_{k=n}^\infty ka_{k}e^{i\theta_k}z^{k-n})^3}}.
	\end{eqnarray*}
	Recall that
	\begin{eqnarray*}
		\RRe w(z) = \RRe (z^n\sum_{k=n}^\infty a_{k}e^{i\theta_k}z^{k-n}) = 0.
	\end{eqnarray*}
	So, letting $v_{n+1}(z):=\sum_{k=n+1}^\infty ka_{k}e^{i\theta_k}{z}^{k-n}$, we have
	\begin{flalign*}
		\RRe&\left(({z}^{n-2}n(n-1)a_{n}e^{i\theta_n})(\bar{z}^{n-1}\sum_{k=n}^\infty ka_{k}e^{-i\theta_k}\bar{z}^{k-n})^2\right)&\\
		=&\RRe\left(({z}^{n-2}n(n-1)a_{n}e^{i\theta_n})(\bar{z}^{n-1}na_ne^{-i\theta_n})^2+({z}^{n-2}n(n-1)a_{n}e^{i\theta_n})2\bar{z}^{n-1}na_ne^{-i\theta_n}\times\right.&\\
		&\left.(\bar{z}^{n-1}\overline{v_{n+1}(z)}) +({z}^{n-2}n(n-1)a_{n}e^{i\theta_n})(\bar{z}^{n-1}\overline{v_{n+1}(z)})^2\right)&\\
		=&\RRe\left(-(\av{z}^{2(n-2)}n^2(n-1)a_n^2\bar{z}^{n}\sum_{k=n+1}^\infty a_{k}e^{-i\theta_k}\bar{z}^{k-n})+({z}^{n-2}n(n-1)a_{n}e^{i\theta_n})2\bar{z}^{n-1}\times\right.&\\
		&\left.na_ne^{-i\theta_n}(\bar{z}^{n-1}\overline{v_{n+1}(z)}) +({z}^{n-2}n(n-1)a_{n}e^{i\theta_n})(\bar{z}^{n-1}\overline{v_{n+1}(z)})^2\right).&
	\end{flalign*}
	Therefore, the remainder term of the numerator in expansion of $\kappa(u)(z)$ is $O(\av{z^{3(n-1)}})$. 
	
	Based on Lemma \ref{L1}, we obtain
	\begin{eqnarray*}
		\av{a_k}\leq 2^{n+2}nk^{2n}\av{a_n}, \forall k\geq n+1.
	\end{eqnarray*}
	So there exists a constant $c(n)$ depending on $n$, such that for all $\av{z}\leq c(n)$,
	\begin{eqnarray*}
		\av{\sum_{k=n}^\infty k(k-1)a_{k}e^{i\theta_k}{z}^{k-n}}\leq C_1(n)\av{a_n},\quad
		\av{v_{n+1}(z)} \leq \frac{\av{a_n}}{2}.
	\end{eqnarray*}
	Thus, there exists constants $C_2(n)$ and $C(n)$ such that
	\begin{eqnarray*}
		\av{\kappa(u)(z)}&\leq&\av{\frac{C_2(n)a_n^3}{(na_n-\frac{\av{a_n}}{2})^3}}\\
		&\leq& C(n),
	\end{eqnarray*}
	which gives the conclusion.\hfill$\square$\par~\\
	
	Now we come to the proof of Theorem \ref{theorem3}. 
	In fact, inspired by the proof of Lemma 3.1 in \cite{logunov2018nodal}, we only need to show that the doubling index of $u$ is finite in $B_{r(n)}$, which is a natural result of growth estimates of Fourier coefficients.~\\
	
	\emph{Proof of Theorem \ref{theorem3}}. By maximum principle and the condition that $u$ has $N$ nodal domains in $B_1(0)$, we have at most $N-1$ nodal curves in $B_1(0)$. With a similar proof in Lemma \ref{l9}, we know that $g = u|_{S^1}$ almost changes its sign on $S^1$ for at most $2N-2$ times. Without loss of generality, we may assume $g = u|_{S^1}$ almost changes its sign on $S^1$ for $2N-2$ times. Furthermore, with growth estimates (\ref{11113}), the following equation holds in the strong sense for any $r<1$: 
	\begin{eqnarray*}
		u(re^{i\theta}) =  \sum_{k=n}^{+\infty}a_kr^k\cos(k\theta + \theta_k),
	\end{eqnarray*}
	where $n\leq N-1$ is the vanishing order of $u$ at $0$. 
	And the frequency of $u$ in $B_r$ is given by
	\begin{eqnarray}\label{55555}
		\beta(r) = \frac{r\int_{B_r}\av{\nabla u}^2}{\int_{\partial B_r}u^2} = \frac{{\sum_{k=n}^{+\infty}ka_k^2 r^{2k}}}{{\sum_{k=n}^{+\infty}a_k^2 r^{2k}}},
	\end{eqnarray}
	see \cite{han2013nodal} for details. 
	It follows from the growth estimates (\ref{11113}) that there exists a constant $r(N)\leq\frac{1}{2}$, such that
	\begin{eqnarray*}
		\sum_{k=N}^{+\infty}ka_k^2 r(N)^{2k}\leq \sum_{k=N}^{+\infty}2^{2N}(N-1)^2k^{4N-3}A_{N-1}^2r(N)^{2k}\leq {r(N)^{2N-2}}A_{N-1}^2,
	\end{eqnarray*}
	where $A_{N-1}=\sup_{n\leq l\leq N-1}\av{a_l}$. 
	With the above estimates, we can derive from (\ref{55555}) that
	\begin{eqnarray*}
		\beta(r(N)) &\leq& \frac{{\sum_{k=n}^{N-1}a_k^2 r(N)^{2k}+{r(N)^{2N-2}}A_{N-1}^2}}{{\sum_{k=n}^{N-1}a_k^2 r(N)^{2k}}}\\
		&=&1+\frac{{r(N)^{2N-2}}A_{N-1}^2}{\sum_{k=n}^{N-1}a_k^2 r(N)^{2k}}\leq 2.
	\end{eqnarray*}
		
	Recall that the doubling index is given by $$2^{\mathcal{N}(r)}:=\frac{\sup_{B_{2r}(0)}\av{u}}{\sup_{B_{r}(0)}\av{u}}.$$
	
	Using the relation between frequency and doubling index (see \cite{logunov2023almost} for details), we have
	\begin{eqnarray*}
		\mathcal{N}(\frac{r(N)}{3}) \leq 2\beta(r(N)) + C\leq 4+C\leq 3C,
	\end{eqnarray*}
	where $C\geq 2$ is a universal constant. This inequality tells us that the growth of $u$ is not rapid in $B_{2r(N)/3}(0)$, that is, 
	\begin{eqnarray}\label{e52}
		\sup_{B_{2r(N)/3}(0)}\av{u}\leq 2^{3C}\sup_{B_{r(N)/3}(0)}\av{u}.
	\end{eqnarray}
	Applying Harnack inequality to $\sup_{B_{r(N)/2}(0)} u \pm u$ and $\sup_{B_{r(N)/2}(0)} (-u) \pm u$, it follows from the fact $u(0)=0$ that
	\begin{eqnarray}\label{e53}
		\sup_{B_{r(N)/3}(0)}\av{u}\leq C_1 \sup_{B_{r(N)/2}(0)}u,\quad \sup_{B_{r(N)/3}(0)}\av{u}\leq C_1 \sup_{B_{r(N)/2}(0)}(-u),
	\end{eqnarray}
	where $C_1\geq 1$ is a universal constant. Let $u(z_1):=\sup_{B_{r(N)/2}(0)}u$, where $z_1\in\partial B_{r(N)/2}(0)$. Combining (\ref{e52}) and (\ref{e53}), we have
	\begin{eqnarray*}
		\sup_{B_{r(N)/6}(z_1)}u\leq\sup_{B_{2r(N)/3}(0)}u\leq 2^{3C}C_1 u(z_1).
	\end{eqnarray*}
	By the standard gradient estimates and Harnack inequality again, we have
	\begin{flalign}\label{e54}
		&&\sup_{B_{r(N)/18}(z_1)}\av{\nabla u}\leq\frac{C_2}{r(N)}\sup_{B_{r(N)/9}(z_1)}\av{u}\leq \frac{C_1C_2}{r(N)}\sup_{B_{r(N)/6}(z_1)}u\leq \frac{2^{3C}C_1^2C_2}{r(N)} u(z_1).
	\end{flalign}
	Thus, if we choose a small universal constant $c=\frac{1}{18\cdot 2^{3C}C_1^2C_2}$, then $u$ keeps positive in $B_{cr(N)}(z_1)$. 
	
	Similarly, letting $u(z_2):=\sup_{B_{r(N)/2}(0)}(-u)$, where $z_2\in\partial B_{r(N)/2}(0)$, we can also derive the inequality (\ref{e54}) for $-u$ in $B_{r(N)/18}(z_2)$. Therefore, we get that $u$ keeps negative in $B_{cr(N)}(z_2)$. Hence,
	\begin{eqnarray*}
		\av{\{u>0\}\cap B_{r(N)}(0)}\geq \av{B_{cr(N)}(z_1)}= \pi c^2r(N)^2, \\
		\av{\{u<0\}\cap B_{r(N)}(0)}\geq \av{B_{cr(N)}(z_2)}= \pi c^2r(N)^2,
	\end{eqnarray*}
	which completes the proof after setting $r_0:=r(N), c_0:=c^2/(1-c^2)$.
	\hfill$\square$\par
	
	\begin{remark}
		In fact, the result in Theorem \ref{theorem3} implies that there is some constant $c_1(N)$, such that
		\begin{eqnarray*}
			c_1(N)\leq\frac{\av{\{u>0\}\cap B_{1}(0)}}{\av{\{u<0\}\cap B_{1}(0)}}\leq c_1(N)^{-1}.
		\end{eqnarray*}
		Nevertheless, we can not find a uniform bound for the above inequality. Actually, there exists an entire function $f$, which is given in \cite{hayman1964meromorphic}, such that $\av{f}<1$ in $\mathbb{C}\backslash D$, where $D$ is a half-strip $\{x+iy\in \mathbb{C}|x>0,\av{y}<\pi\}$. Since $f$ is an entire function, there is some point $z_0\in D$ such that $\RRe f(z_0)>1$. Therefore, the level set of $u(z)=\RRe f$ at $z_0$ is given by $\{z\in\mathbb{C}|u(z)=u(z_0)\}\subset D$. And for $z\notin D$, $u(z)<u(z_0)$. Thus, for $R\rightarrow+\infty$, 
		\begin{eqnarray*}
			\frac{\av{\{u>u(z_0)\}\cap B_{R}(z_0)}}{\av{\{u<u(z_0)\}\cap B_{R}(z_0)}}\rightarrow 0,
		\end{eqnarray*}
		which means $c_1$ can not be universal.
	\end{remark}
	
	\section{Applications of M\"{o}bius transformation}
	In fact, the point where nodal curves intersect is not crucial, since for any point $p\in B_1(0)$, there is a M\"{o}bius transformation $\psi:B_1(0)\rightarrow B_1(0)$ such that $\psi(p)=0$ and $\psi$ maintains the unit disk. Thus, we have the following theorem, which is a generalization of Theorem \ref{Main_theorem}.
	\begin{theorem}\label{t14}
		Let $u:B_{r_0}(0)\subseteq \R^2 \rightarrow \R$ be a non-constant harmonic function with boundary value $u|_{S^1_{r_0}}\in L^1(S^1_{r_0})$, such that $u(p)=0$, where $p\in B_{r_0}(0)$. Assume $u$ has only $n$ nodal curves that intersect at $p$ in $B_{r_0}(0)$.
		Then for every nodal curve $\gamma_k\subset Z(u)$,
		\begin{eqnarray}\label{mobiusestimate}
			\bigg|\kappa(u)|_{\gamma_k}(p)\bigg|\leq \frac{4(n+1)}{(r_0^2-\av{p}^2)n}r_0\cos n\alpha_0+\frac{2\av{p}}{r_0^2-\av{p}^2},
		\end{eqnarray}
		where $\alpha_0=0$ for odd $n$ and $\alpha_0=\frac{\pi}{2n(n+1)}$ for even $n$. 
		Moreover, this equality is sharp for all $n\geq 1$.
	\end{theorem}
	\begin{proof}
		First, assume $r_0=1$, and $u=\RRe w$, such that $w(p)=w^\prime(p)=\cdots=w^{(n-1)}(p)=0$ and $w^{(n)}(p)\neq 0$. A M\"{o}bius transformation $\psi$, mapping $p$ to $0$, is defined as
		\begin{eqnarray}\label{representation}
			\psi(z)=\frac{z-p}{1-\overline{p}z}e^{i\theta},
		\end{eqnarray}
		where $\theta$ is a constant to be determined. Let $w_1(z):=w(\psi^{-1}(z))$ and $u_1(z):=\RRe w_1=u(\psi^{-1}(z))$. Since $\psi$ is a conformal map, $u_1$ also has only $n$ nodal curves in $B_1(0)$, intersecting at $0$. Moreover, if we define $\psi(e^{is}):=e^{it(s)}$, then 
		$$\derivative{e^{it(s)}}{s}=it^\prime(s)\psi(e^{is})=i\psi^\prime(e^{is})e^{is}=ie^{is}\frac{1-\av{p}^2}{(1-\overline{p}e^{is})^2}=i\psi(e^{is})\frac{1-\av{p}^2}{\av{1-\overline{p}e^{is}}^2}.$$
		Thus, $t^\prime(s)=(1-\av{p}^2)/\av{1-\overline{p}e^{is}}^2$, which indicates that $u_1|_{S^1}\in L^1(S^1)$ since $\av{p}<1$. Hence, $u_1$ satisfies the conditions of Theorem \ref{Main_theorem}, which gives
		\begin{eqnarray*}
			\bigg|\kappa(u_1)|_{\widetilde{\gamma}_k}(0)\bigg|\leq \frac{4(n+1)}{n}\cos n\alpha_0,
		\end{eqnarray*}
		where $\widetilde{\gamma}_k=\psi^{-1}(\gamma_k)$, and $\alpha_0=0$ for odd $n$ and $\alpha_0=\frac{\pi}{2n(n+1)}$ for even $n$. Now for $w(z)=w_1(\psi(z))$, we have
		\begin{eqnarray*}
			w^{\prime}(z) = w_1^{\prime}(\psi(z))\psi^{\prime}(z);\ 
			w^{\prime\prime}(z) = w_1^{\prime\prime}(\psi(z))\psi^{\prime}(z)^2+w_1^{\prime}(\psi(z))\psi^{\prime\prime}(z).
		\end{eqnarray*}
		By induction, we know that for any integer $l\geq 2$, $w^{(l)}(z)=w_1^{(l)}(\psi(z))\psi^{\prime}(z)^{l}+\left(\begin{matrix}
			l\\2
		\end{matrix}\right)w_1^{(l-1)}(\psi(z))\psi^{\prime}(z)^{l-2}\psi^{\prime\prime}(z)+\cdots$, where we omit terms with less than $(l-1)$th derivative. It follows from the condition $w(p)=w^\prime(p)=\cdots=w^{(n-1)}(p)=0$ that 
		\begin{eqnarray*}
			w^{(n)}(p)&=&w_1^{(n)}(0)\psi^{\prime}(p)^{n},\\ 
			w^{(n+1)}(p)&=&w_1^{(n+1)}(0)\psi^{\prime}(p)^{n+1}+\left(\begin{matrix}
				n+1\\2
			\end{matrix}\right)w_1^{(n)}(0)\psi^{\prime}(p)^{n-1}\psi^{\prime\prime}(p).
		\end{eqnarray*}
		Since $\psi^{\prime}(z)=e^{i\theta}(1-\av{p}^2)/(1-\overline{p}z)^2$, we obtain
		\begin{eqnarray*}
			\psi^{\prime}(p)=\frac{e^{i\theta}}{1-\av{p}^2},\ 
			\psi^{\prime\prime}(p)=\frac{2\overline{p}e^{i\theta}}{\left(1-\av{p}^2\right)^2}.
		\end{eqnarray*}
		
		After choosing $\theta=arg(w^{(n)}(p))/n$, we get $arg(w_1^{(n)}(0))=0$, that is, $w_1^{(n)}(0)>0$. Furthermore, $\theta+\eta_q=\frac{q\pi+\frac{1}{2}\pi}{n}$.
		
		Therefore, with equations (\ref{qqqqqq}) and (\ref{000}), we can calculate $k(u)(p)$ along $\gamma_k$ as
		\begin{eqnarray}\label{mobius}
			&&\bigg|\kappa(u)|_{{\gamma}_k}(p)\bigg|
			=\av{\frac{2}{n^2+n}\frac{\RRe\left(e^{i(n+1)\eta_q}w^{(n+1)}(p)\right)}{\av{w^{(n)}(p)}}}\nonumber\\
			&&= \av{\frac{2\RRe\left(e^{i(n+1)\eta_q}(w_1^{(n+1)}(0)(\frac{e^{i\theta}}{1-\av{p}^2})^{n+1}+\left(\begin{matrix}
						n+1\\2
					\end{matrix}\right)w_1^{(n)}(0)\frac{2\overline{p}e^{in\theta}}{\left(1-\av{p}^2\right)^{n+1}})\right)}{n(n+1)\av{w_1^{(n)}(0)(\frac{e^{i\theta}}{1-\av{p}^2})^{n}}}}\nonumber\\
			&&=\av{\frac{2\RRe\left(e^{i(n+1)\eta_q}(w_1^{(n+1)}(0)e^{i(n+1)\theta}+n(n+1)w_1^{(n)}(0)\overline{p}e^{in\theta})\right)}{n(n+1)\left(1-\av{p}^2\right)\av{w_1^{(n)}(0)}}}\nonumber\\
			&&=\av{\frac{2\RRe\left(e^{i(n+1)(\theta+\eta_q)}w_1^{(n+1)}(0)\right)}{n(n+1)\left(1-\av{p}^2\right)\av{w_1^{(n)}(0)}}+\frac{2\RRe\left(\overline{p}e^{i((n+1)\eta_q+n\theta)}\right)}{1-\av{p}^2}}\nonumber\\
			&&= \av{\frac{2\RRe\left(e^{i(n+1)\frac{q\pi+\frac{1}{2}\pi}{n}}w_1^{(n+1)}(0)\right)}{n(n+1)\left(1-\av{p}^2\right)\av{w_1^{(n)}(0)}}+\frac{2\RRe\left(\overline{p}e^{i\frac{(n+1)(q\pi+\frac{1}{2}\pi)-arg(w^{(n)}(0))}{n}}\right)}{1-\av{p}^2}}\nonumber\\
			&&= \av{\frac{-\kappa(u_1)|_{\widetilde{\gamma}_k}(0)}{1-\av{p}^2}+\frac{2\RRe\left(\overline{p}e^{i\frac{(n+1)(q\pi+\frac{1}{2}\pi)-arg(w^{(n)}(0))}{n}}\right)}{1-\av{p}^2}}\nonumber\\
			&&\leq 
			\av{\frac{\kappa(u_1)|_{\widetilde{\gamma}_k}(0)}{1-\av{p}^2}}+\frac{2\av{p}}{1-\av{p}^2}
			\leq\frac{4(n+1)}{(1-\av{p}^2)n}\cos n\alpha_0+\frac{2\av{p}}{1-\av{p}^2}.
		\end{eqnarray}
		
		For general $r_0>0$, let $\widetilde{u}(z):=u(r_0z)$ and $\widetilde{w}(z):=w(r_0z)$, which are defined in $B_1(0)$. Then $\widetilde{u}(z)$ satisfies the inequality (\ref{mobius}). Thus, we have
		\begin{eqnarray*}
			\bigg|\kappa(u)|_{{\gamma}_k}(p)\bigg| &=& \av{\frac{2}{n^2+n}\frac{\RRe\left(e^{i(n+1)\eta_q}w^{(n+1)}(p)\right)}{\av{w^{(n)}(p)}}}\\
			&=& \av{\frac{2}{n^2+n}\frac{\RRe\left(e^{i(n+1)\eta_q}\widetilde{w}^{(n+1)}(\frac{p}{r_0})\right)}{r_0\av{\widetilde{w}^{(n)}(\frac{p}{r_0})}}}\\
			&\leq&\frac{4(n+1)}{(r_0^2-\av{p}^2)n}r_0\cos n\alpha_0+\frac{2\av{p}}{r_0^2-\av{p}^2},
		\end{eqnarray*}
		which gives the desired estimate. 
		
		Moreover, according to Lemma \ref{sharpness}, the last inequality in (\ref{mobius}) is sharp. 
		Assume for some integer $j$, $(-1)^{j-1}\kappa(u_1)|_{\widetilde{\gamma}_k}(0)\geq 0$. Then the equality holds in the first inequality in (\ref{mobius}) if and only if 
		\begin{eqnarray*}
			-arg(p)+\frac{(n+1)(q\pi+\frac{1}{2}\pi)-arg(w^{(n)}(p))}{n} = j\pi,
		\end{eqnarray*}
		where $arg(p)$ denotes the argument of $p$. Recall that $\theta=arg(w^{(n)}(p))/n$. Hence, to let the equality hold, we only need to choose 
		\begin{eqnarray}\label{equality}
			\theta = -j\pi -arg(p)+\frac{(n+1)(q\pi+\frac{1}{2}\pi)}{n} ,
		\end{eqnarray}
		in equation (\ref{representation}) to construct some appropriate function $w(z)$. In conclusion, the estimate (\ref{mobiusestimate}) is sharp for all $n\geq 1$.		
	\end{proof}
	\begin{remark}
		In fact, the equality holds in the last inequality in (\ref{mobius}) if and only if $w_1$ is given by (\ref{extremers}). Therefore, the extremer of Theorem \ref{t14}, when the equality holds in (\ref{mobiusestimate}), is given by
		\begin{eqnarray*}
			u(z) = \RRe w(\frac{z-p}{1-\overline{p}z}e^{i\theta}),
		\end{eqnarray*}
		where $w$ is given by (\ref{extremers}) and $\theta$ is given by (\ref{equality}).
	\end{remark}
	
	Similarly, through a M\"{o}bius transformation $\psi$, we can establish a connection between $u(p)=0$ and $u_1(0)=0$ by $u(z)=u_1(\psi(z))$. As a result of inequality (\ref{mobius}), we can generalize Theorem \ref{theorem2} as follows, where we omit the proof.
	\begin{corollary}\label{C1}
		Let $u:B_{r_0}(0)\subseteq \R^2 \rightarrow \R$ be a non-constant harmonic function with the boundary value $u|_{S^1_{r_0}}\in L^1(S^1_{r_0})$, such that $u(p)=0$, where $p\in B_{r_0}(0)$. Assume $u$ has only $n=n(u,p)$ nodal curves intersecting at $p$ in $B_{r_0}(0)$.
		Then there is a constant $c(n)<1$ and $C(n)$, depending on $n$, such that
		\begin{eqnarray*}
			\bigg|\kappa(u)(z)\bigg|\leq \frac{C(n)r_0}{(r_0^2-\av{p}^2)}, \qquad \forall z\in (Z(u)\backslash\{p\})\cap \psi^{-1}(B_{c(n)}(0)),
		\end{eqnarray*}
		where $\psi$ is given by (\ref{representation}).
	\end{corollary}
	\begin{remark}
		Note that for $z\in B_{1}(0)$, 
		\begin{eqnarray*}
			\av{z}(1-\av{p})\leq \av{\psi^{-1}(z)-p}=\av{ze^{-i\theta}\frac{1-\av{p}^2}{1+\overline{p}ze^{-i\theta}}}\leq \av{z}(1+\av{p}),
		\end{eqnarray*}
		which means $B_{(1-\av{p})r_0}(p)\subset\psi^{-1}(B_{r_0}(0))$. Thus, the conclusions of Corollary \ref{C1} holds in a small disk with center $p$.
	\end{remark}
	
	Besides, any M\"{o}bius transformation $\psi$ is a conformal map, which implies that after a transformation $u(z)=u_1(\psi(z))$, the number of nodal domains of $u$ is the same as that of $u_1$. Therefore, we can generalize Theorem \ref{theorem3} as follows.
	\begin{corollary}\label{C2}
		Let $u:B_1(0)\subseteq \R^2 \rightarrow \R$ be a non-constant harmonic function satisfying $u(p)=0$ with $u|_{S^1}\in L^1(S^1)$, and $u$ has $N$ nodal domains in $B_1(0)$, then 
		\begin{eqnarray*}
			c_0(p)\leq\frac{\av{\{u>0\}\cap \psi^{-1}(B_{r_0}(0))}}{\av{\{u<0\}\cap \psi^{-1}(B_{r_0}(0))}}\leq c_0(p)^{-1},
		\end{eqnarray*}
		where $c_0(p)$ is a constant depending on $p$, $r_0=r_0(N)<1/2$ depends only on $N$, and $\psi$ is given by (\ref{representation}).
	\end{corollary}
	\begin{proof}
		Note that for $u_1(z)=u(\psi^{-1}(z))$, $u_1$ satisfies the condition in Theorem \ref{theorem3}. Hence, according to the proof of Theorem \ref{theorem3}, $u_1(z)$ keeps positive in $B_{cr(N)}(z_1)$ and negative in $B_{cr(N)}(z_2)$, where $z_1, z_2\in \partial B_{r(N)/2}(0)$. Thus, $u(z)$ keeps positive in $\psi^{-1}(B_{cr(N)}(z_1))$ and negative in $\psi^{-1}(B_{cr(N)}(z_2))$. Note that for all $z\in B_1(0)$,
		\begin{eqnarray*}
			\av{(\psi^{-1})^{\prime}(z)}=\av{e^{-i\theta}\frac{1-\av{p}^2}{(1+\overline{p}ze^{-i\theta})^2}}\geq \frac{1-\av{p}}{1+\av{p}},
		\end{eqnarray*}
		which gives 
		\begin{eqnarray*}
			\av{\{u>0\}\cap \psi^{-1}(B_{r(N)}(0))}&\geq& \av{\psi^{-1}(B_{cr(N)}(z_1))}\\
			&=&\int_{B_{cr(N)}(z_1)}\av{(\psi^{-1})^{\prime}(z)}^2\\
			&\geq&\pi c^2r(N)^2\left(\frac{1-\av{p}}{1+\av{p}}\right)^2.
		\end{eqnarray*}
		The estimate for $\{u<0\}\cap \psi^{-1}(B_{r(N)}(0))$ is just the same. 
		The conclusion is established after setting 
		$$r_0:=r(N), c_0(p):=\frac{c^2\left({1-\av{p}}\right)^2}{\left({1+\av{p}}\right)^2-c^2\left({1-\av{p}}\right)^2}.$$
	\end{proof}

	\section{Auxiliary lemmas}
	\begin{lemma}\label{auxilary}
		The linear equations (\ref{inverse}) is solvable.
	\end{lemma}
	\begin{proof}
		It is sufficient to show that
		\begin{eqnarray}\label{nonzero}
			\left|\begin{matrix}
				\beta_2-\beta_1  &\cdots & \beta_{2n+1}-\beta_{2n}\\
				\sin\beta_2-\sin\beta_1  & \cdots & \sin\beta_{2n+1}-\sin\beta_{2n}\\
				\cos\beta_2-\cos\beta_1  & \cdots & \cos\beta_{2n+1}-\cos\beta_{2n}\\
				\vdots  & \vdots & \vdots \\
				\sin n\beta_2-\sin n\beta_1  & \cdots & \sin n\beta_{2n+1}-\sin n\beta_{2n}
			\end{matrix}\right|\neq 0.
		\end{eqnarray} 
		Adding all other columns to the last column and noticing $\beta_{2n+1}=\beta_1+2\pi$, the last column becomes $(2\pi,0,\ldots,0)^\prime$. Then it is sufficient to show that
		\begin{eqnarray*}
			\left|\begin{matrix}
				\sin\beta_2-\sin\beta_1  & \cdots & \sin\beta_{2n}-\sin\beta_{2n-1}\\
				\cos\beta_2-\cos\beta_1  & \cdots & \cos\beta_{2n}-\cos\beta_{2n-1}\\
				\vdots  & \vdots & \vdots \\
				\sin n\beta_2-\sin n\beta_1  & \cdots & \sin n\beta_{2n}-\sin n\beta_{2n-1}
			\end{matrix}\right|\neq 0.
		\end{eqnarray*} 
		Recall $\beta_j:=(\alpha_j+\alpha_{j+1})/2$ and  $\alpha_{j}=\alpha_0+(j-n-1)\varepsilon_0$. Then we can get that $\sin k\beta_{j}-\sin k\beta_{j-1}=2\cos (k(\beta_{j}+\beta_{j-1})/2)\sin(k(\beta_{j}-\beta_{j-1})/2)=2\cos (k\alpha_j)\sin(k\varepsilon_0/2)$. Similarly,  $\cos k\beta_{j}-\cos k\beta_{j-1}=-2\sin (k\alpha_j)\sin(k\varepsilon_0/2)$. Therefore, after multiplying a non-zero constant, (\ref{nonzero}) is equivalent to
		\begin{eqnarray*}
		 	\left|\begin{matrix}
		 		\cos \alpha_2  & \cdots & \cos\alpha_{2n}\\
		 		\sin\alpha_2  & \cdots & \sin\alpha_{2n}\\
		 		\vdots  & \vdots & \vdots \\
		 		\cos n\alpha_2  & \cdots & \cos n\alpha_{2n}
		 	\end{matrix}\right|=
		 	\left|\begin{matrix}
		 		\frac{e^{i\alpha_2}+e^{-i\alpha_2}}{2}   & \cdots & \frac{e^{i\alpha_{2n}}+e^{-i\alpha_{2n}}}{2}\\
		 		\frac{e^{i\alpha_2}-e^{-i\alpha_2}}{2i}   & \cdots & \frac{e^{i\alpha_{2n}}-e^{-i\alpha_{2n}}}{2i}\\
		 		\vdots  & \vdots & \vdots \\
		 		\frac{e^{in\alpha_2}+e^{-in\alpha_2}}{2}   & \cdots & \frac{e^{in\alpha_{2n}}+e^{-in\alpha_{2n}}}{2}
		 	\end{matrix}\right|\neq 0,
		\end{eqnarray*} 
		which is equivalent to say
		\begin{eqnarray*}
					0	&\neq&\RRe \left(i^{n-1}\left|\begin{matrix}
				e^{i\alpha_2}   & \cdots & e^{i\alpha_{2n}}\\
				e^{-i\alpha_2}   & \cdots & e^{-i\alpha_{2n}}\\
				\vdots  & \vdots & \vdots \\
				e^{in\alpha_2}   & \cdots & e^{in\alpha_{2n}}
			\end{matrix}\right|\right)\\
			&=&\RRe \left((-1)^{(n-1)^2}i^{n-1}\prod_{j=2}^{2n}e^{-i(n-1)\alpha_j}\left|\begin{matrix}
				1  & \cdots & 1\\
				e^{i\alpha_2}   & \cdots & e^{i\alpha_{2n}}\\
				\vdots  & \vdots & \vdots \\
				e^{i(n-2)\alpha_2}   & \cdots & e^{i(n-2)\alpha_{2n}}\\
				e^{in\alpha_2}   & \cdots & e^{in\alpha_{2n}}\\
				\vdots  & \vdots & \vdots \\
				e^{i(2n-1)\alpha_2}   & \cdots & e^{i(2n-1)\alpha_{2n}}
			\end{matrix}\right|\right).
		\end{eqnarray*} 
		By similar calculation with Vandermonde determinant, we have 
		\begin{eqnarray*}
			\left|\begin{matrix}
				1  & \cdots & 1\\
				e^{i\alpha_2}   & \cdots & e^{i\alpha_{2n}}\\
				\vdots  & \vdots & \vdots \\
				e^{i(n-2)\alpha_2}   & \cdots & e^{i(n-2)\alpha_{2n}}\\
				e^{in\alpha_2}   & \cdots & e^{in\alpha_{2n}}\\
				\vdots  & \vdots & \vdots \\
				e^{i(2n-1)\alpha_2}   & \cdots & e^{i(2n-1)\alpha_{2n}}
			\end{matrix}\right|=\sigma_{n}(e^{i\alpha_2},\ldots,e^{i\alpha_{2n}})\prod_{2\leq j_1< j_2\leq 2n}(e^{i\alpha_{j_2}}-e^{i\alpha_{j_1}}),
		\end{eqnarray*} 
		where $\sigma_{n}$ represents $n$-th elementary symmetric polynomial. Thus, by Taylor expansion, we obtain
		$$
		\begin{aligned}
			&\RRe \left(i^{n-1}\prod_{j=2}^{2n}e^{-i(n-1)\alpha_j}\left|\begin{matrix}
				1  & \cdots & 1\\
				e^{i\alpha_2}   & \cdots & e^{i\alpha_{2n}}\\
				\vdots  & \vdots & \vdots \\
				e^{i(n-2)\alpha_2}   & \cdots & e^{i(n-2)\alpha_{2n}}\\
				e^{in\alpha_2}   & \cdots & e^{in\alpha_{2n}}\\
				\vdots  & \vdots & \vdots \\
				e^{i(2n-1)\alpha_2}   & \cdots & e^{i(2n-1)\alpha_{2n}}
			\end{matrix}\right|\right)\\
			&=\RRe
			\left(i^{n-1}\prod_{j=2}^{2n}e^{-i(n-1)\alpha_j}\sigma_{n}(e^{i\alpha_2},\ldots,e^{i\alpha_{2n}})\prod_{2\leq j_1< j_2\leq 2n}(e^{i\alpha_{j_2}}-e^{i\alpha_{j_1}})\right)\\
			&=\RRe
			\left(i^{n-1}e^{-i(2n-1)(n-1)\alpha_0}\left(\begin{matrix}
				2n-1 \\
				n
			\end{matrix}\right)e^{in\alpha_0}(i\varepsilon_0e^{i\alpha_0})^{(2n-1)(n-1)}\prod_{2\leq j_1< j_2\leq 2n}(j_2-j_1)\right)\\
			&\quad+\ o(\varepsilon_0^{(2n-1)(n-1)})\\
			&=\left(\begin{matrix}
				2n-1 \\
				n
			\end{matrix}\right)\RRe
			\left(e^{in\alpha_0}\right)\varepsilon_0^{(2n-1)(n-1)}\prod_{2\leq j_1< j_2\leq 2n}(j_2-j_1)+\ o(\varepsilon_0^{(2n-1)(n-1)}),
		\end{aligned} 
		$$
		where  $\alpha_0:=0$ for odd $n$ and $\alpha_0:=\frac{\pi}{2n(n+1)}$ for even $n$. For all $n\geq 1$, the first term is always positive, so for all small $\varepsilon_0$, (\ref{nonzero}) always holds.
	\end{proof}

	\section*{Acknowledgments}
	We thank Stefan Steinerberger for helpful discussions and bringing our attention to \"{U}. Kuran's paper \cite{K}. We thank Bobo Hua for his helpful suggestions and constant support. We also thank Xingjian Hu for sincere help. The author is supported by Shanghai Science and Technology Program [Project No. 22JC1400100].
	
	\bibliography{refer}                        

\begin{bibdiv}
\begin{biblist}

\bib{AG}{article}{
      author={Alessandrini, Giovanni},
       title={The length of level lines of solutions of elliptic equations in
  the plane},
        date={1988},
        ISSN={0003-9527},
     journal={Arch. Rational Mech. Anal.},
      volume={102},
      number={2},
       pages={183\ndash 191},
         url={https://doi.org/10.1007/BF00251498},
      review={\MR{943431}},
}

\bib{BBJ}{article}{
      author={Bian, BJ},
      author={Guan, P},
      author={Ma, XN},
      author={Xu, L},
       title={A microscopic convexity principle for the level sets of solution
  for nonlinear elliptic partial differential equations},
        date={2011},
     journal={Indiana Univ. Math. J},
      volume={60},
      number={1},
       pages={101\ndash 119},
}

\bib{BM}{article}{
      author={Boothby, William~M.},
       title={The topology of the level curves of harmonic functions with
  critical points},
        date={1951},
        ISSN={0002-9327},
     journal={Amer. J. Math.},
      volume={73},
       pages={512\ndash 538},
         url={https://doi.org/10.2307/2372305},
      review={\MR{43456}},
}

\bib{CA}{article}{
      author={Caffarelli, Luis~A.},
      author={Friedman, Avner},
       title={Convexity of solutions of semilinear elliptic equations},
        date={1985},
        ISSN={0012-7094},
     journal={Duke Math. J.},
      volume={52},
      number={2},
       pages={431\ndash 456},
         url={https://doi.org/10.1215/S0012-7094-85-05221-4},
      review={\MR{792181}},
}

\bib{CAMY}{article}{
      author={Chang, Sun-Yung~Alice},
      author={Ma, Xi-Nan},
      author={Yang, Paul},
       title={Principal curvature estimates for the convex level sets of
  semilinear elliptic equations},
        date={2010},
        ISSN={1078-0947},
     journal={Discrete Contin. Dyn. Syst.},
      volume={28},
      number={3},
       pages={1151\ndash 1164},
         url={https://doi.org/10.3934/dcds.2010.28.1151},
      review={\MR{2644784}},
}

\bib{CH}{article}{
      author={De~Carli, L.},
      author={Hudson, S.~M.},
       title={Geometric remarks on the level curves of harmonic functions},
        date={2010},
        ISSN={0024-6093},
     journal={Bull. Lond. Math. Soc.},
      volume={42},
      number={1},
       pages={83\ndash 95},
         url={https://doi.org/10.1112/blms/bdp099},
      review={\MR{2586969}},
}

\bib{KB}{article}{
      author={Diaz, J.~I.},
      author={Kawohl, B.},
       title={On convexity and starshapedness of level sets for some nonlinear
  elliptic and parabolic problems on convex rings},
        date={1993},
        ISSN={0022-247X},
     journal={J. Math. Anal. Appl.},
      volume={177},
      number={1},
       pages={263\ndash 286},
         url={https://doi.org/10.1006/jmaa.1993.1257},
      review={\MR{1224819}},
}

\bib{DNJ}{article}{
      author={Dorevaar, Nichlas~J},
       title={Conexity of level sets for solutions to elliptic ring problems},
        date={1990},
     journal={Communications in Partial Differential Equations},
      volume={15},
      number={4},
       pages={541\ndash 556},
}

\bib{EAD}{article}{
      author={Enciso, Alberto},
      author={Peralta-Salas, Daniel},
       title={Some geometric conjectures in harmonic function theory},
        date={2013},
        ISSN={0373-3114},
     journal={Ann. Mat. Pura Appl. (4)},
      volume={192},
      number={1},
       pages={49\ndash 59},
         url={https://doi.org/10.1007/s10231-011-0211-4},
      review={\MR{3011323}},
}

\bib{F}{article}{
      author={Flatto, L.},
       title={A theorem on level curves of harmonic functions},
        date={1969},
        ISSN={0024-6107},
     journal={J. London Math. Soc. (2)},
      volume={1},
       pages={470\ndash 472},
         url={https://doi.org/10.1112/jlms/s2-1.1.470},
      review={\MR{247118}},
}

\bib{FDH}{article}{
      author={Flatto, Leopold},
      author={Newman, Donald~J.},
      author={Shapiro, Harold~S.},
       title={The level curves of harmonic functions},
        date={1966},
        ISSN={0002-9947},
     journal={Trans. Amer. Math. Soc.},
      volume={123},
       pages={425\ndash 436},
         url={https://doi.org/10.2307/1994665},
      review={\MR{197755}},
}

\bib{GRM}{article}{
      author={Gabriel, R.~M.},
       title={A result concerning convex level surfaces of {$3$}-dimensional
  harmonic functions},
        date={1957},
        ISSN={0024-6107},
     journal={J. London Math. Soc.},
      volume={32},
       pages={286\ndash 294},
         url={https://doi.org/10.1112/jlms/s1-32.3.286},
      review={\MR{90662}},
}

\bib{grafakos2008classical}{book}{
      author={Grafakos, Loukas},
      author={others},
       title={Classical fourier analysis},
   publisher={Springer},
        date={2008},
      volume={2},
}

\bib{han2013nodal}{article}{
      author={Han, Qing},
      author={Lin, Fang-Hua},
       title={Nodal sets of solutions of elliptic differential equations},
        date={2013},
     journal={Book in preparation},
}

\bib{hayman1964meromorphic}{article}{
      author={Hayman, Walter~Kurt},
       title={Meromorphic functions},
        date={1964},
     journal={Oxford Math. Monogr.},
}

\bib{RPJ}{article}{
      author={Jerrard, R.~P.},
      author={Rubel, L.~A.},
       title={On the curvature of the level lines of a harmonic function},
        date={1963},
        ISSN={0002-9939},
     journal={Proc. Amer. Math. Soc.},
      volume={14},
       pages={29\ndash 32},
         url={https://doi.org/10.2307/2033949},
      review={\MR{142770}},
}

\bib{JJMO}{article}{
      author={Jost, J\"{u}rgen},
      author={Ma, Xi-Nan},
      author={Ou, Qianzhong},
       title={Curvature estimates in dimensions 2 and 3 for the level sets of
  {$p$}-harmonic functions in convex rings},
        date={2012},
        ISSN={0002-9947},
     journal={Trans. Amer. Math. Soc.},
      volume={364},
      number={9},
       pages={4605\ndash 4627},
         url={https://doi.org/10.1090/S0002-9947-2012-05436-5},
      review={\MR{2922603}},
}

\bib{KW}{article}{
      author={Kaplan, Wilfred},
       title={Topology of level curves of harmonic functions},
        date={1948},
        ISSN={0002-9947},
     journal={Trans. Amer. Math. Soc.},
      volume={63},
       pages={514\ndash 522},
         url={https://doi.org/10.2307/1990572},
      review={\MR{25159}},
}

\bib{K}{article}{
      author={Kuran, {\"U}},
       title={On the zeros of harmonic functions},
        date={1969},
     journal={Journal of the London Mathematical Society},
      volume={1},
      number={1},
       pages={303\ndash 309},
}

\bib{HL}{article}{
      author={Lin, Fang-Hua},
       title={Nodal sets of solutions of elliptic and parabolic equations},
        date={1991},
        ISSN={0010-3640},
     journal={Comm. Pure Appl. Math.},
      volume={44},
      number={3},
       pages={287\ndash 308},
         url={https://doi.org/10.1002/cpa.3160440303},
      review={\MR{1090434}},
}

\bib{logunov2018nodal}{incollection}{
      author={Logunov, Alexander},
      author={Malinnikova, Eugenia},
       title={Nodal sets of laplace eigenfunctions: estimates of the hausdorff
  measure in dimensions two and three},
        date={2018},
   booktitle={50 years with hardy spaces: A tribute to victor havin},
   publisher={Springer},
       pages={333\ndash 344},
}

\bib{logunov2023almost}{article}{
      author={Logunov, Alexander},
      author={Priya~ME, Lakshmi},
      author={Sartori, Andrea},
       title={Almost sharp lower bound for the nodal volume of harmonic
  functions},
        date={2023},
     journal={Communications on Pure and Applied Mathematics},
}

\bib{ramm2022dirichlet}{article}{
      author={Ramm, Alexander~G},
       title={Dirichlet problem with l 1 (s) boundary values},
        date={2022},
     journal={Axioms},
      volume={11},
      number={8},
       pages={371},
}

\bib{SP}{article}{
      author={Salani, Paolo},
       title={Starshapedness of level sets of solutions to elliptic {PDE}s},
        date={2005},
        ISSN={0003-6811},
     journal={Appl. Anal.},
      volume={84},
      number={12},
       pages={1185\ndash 1197},
         url={https://doi.org/10.1080/00036810412331297262},
      review={\MR{2178766}},
}

\bib{SS}{article}{
      author={Steinerberger, Stefan},
       title={On the curvature of level sets of harmonic functions},
        date={2013},
     journal={arXiv preprint arXiv:1307.2069},
}

\bib{strichartz2003guide}{book}{
      author={Strichartz, Robert~S},
       title={A guide to distribution theory and fourier transforms},
   publisher={World Scientific Publishing Company},
        date={2003},
}

\bib{WWZ}{incollection}{
      author={Wen, Zhi~Ying},
      author={Wu, Li~Ming},
      author={Zhang, Yiping},
       title={Set of zeros of harmonic functions of two variables},
        date={1991},
   booktitle={Harmonic analysis ({T}ianjin, 1988)},
      series={Lecture Notes in Math.},
      volume={1494},
   publisher={Springer, Berlin},
       pages={196\ndash 203},
         url={https://doi.org/10.1007/BFb0087772},
      review={\MR{1187080}},
}

\end{biblist}
\end{bibdiv}

\end{document}